\documentclass[reqno,12pt]{amsart} 
\usepackage{vmargin}
\setpapersize{A4}

\usepackage{amsmath} 

\usepackage{amsfonts}   
\usepackage{color}
   
\usepackage{amssymb}      

\usepackage{eufrak}      

\usepackage{amscd}      
\usepackage{amsthm}      

\usepackage{tikz,amsmath}
\usetikzlibrary{arrows}

\usepackage{epsfig}      

\usepackage{amstext}      

\usepackage{graphicx}

\usepackage[all,line]{xy} 
\CompileMatrices 

   \newcommand{\Hom}{\operatorname{Hom}}

\newcommand{\Int}{\operatorname{Int}}

   \theoremstyle{plain}
   \newtheorem{thm}{Theorem}[section]
   \newtheorem{prop}[thm]{Proposition}
   \newtheorem{lemma}[thm]{Lemma}  
   \newtheorem{cor}[thm]{Corollary}
   \theoremstyle{defn}
   
   \newtheorem{defn}[thm]{Definition}
   
   \theoremstyle{remark}

\newtheorem{assumption}[thm]{Assumption}

\usepackage{lipsum}

   \numberwithin{equation}{section}

        \date{\today}

\title[KMS and $\mathrm{KMS}_{\infty}$]{Equilibrium and ground states from Cayley graphs}  
\author{Johannes Christensen}

\author{Klaus Thomsen}

\date{\today}

\email{johannes@math.au.dk; matkt@math.au.dk}
\address{Institut for Matematik, Aarhus University, Ny Munkegade, 8000 Aarhus C, Denmark}

\begin{document}

\maketitle

\section{Introduction} The introduction of graph $C^*$-algebras more than 20 years ago by Pask, Kumjian, Raeburn and Renault, \cite{KPRR}, initiated many exciting developments in $C^*$-algebra theory and gave opportunities for new connections to other fields of mathematics. One such is the connection to geometric group theory coming from Cayley graphs and the main purpose with this paper is to present a first investigation of the quantum statistical models which arise by considering generalized gauge actions on the $C^*$-algebra of a Cayley graph. From previous work involving infinite graphs it has become clear that the study of KMS states and weights for generalized gauge actions on the $C^*$-algebra of a infinite graph is closely related to the theory of denumerable Markov chains and random walks, and that results from these fields can provide solutions to some of the questions motivated by the physical interpretations. In many cases, however, these questions turn into problems that are notoriuosly very difficult, and only little help is offered by the results from other fields of mathematics. This is intriguing because the results in \cite{Th1} show that infinite graphs offer a richness in the structure of KMS states and weights which can not be realized with finite graphs and which is only paralleled by the constructions made by Bratteli, Elliott and Kishimoto in the 80's, cf. \cite{BEK}. Furthermore, several of the issues that arise from the operator algebra setting and its interpretation as a model in quantum statistical mechanics have not, or only very marginally been considered from the point of view of Markov chains or random walks and they call for new ideas.         

In this paper we set up the general framework for the study of KMS states of generalized gauge actions on the graph $C^*$-algebra of a Cayley graph, or more precisely the restriction of these actions to the corner of the algebra obtained by considering the neutral element of the group as a distinguished vertex. These states are in one-to-one correspondence with the vectors that are normalized and harmonic for a matrix over the group which depends on the inverse temperature $\beta$. For fixed $\beta$ one can in this way translate the problem of finding the $\beta$-KMS states to one which deals with a stochastic matrix and hence with a random walk on the group. This opens up for the possible exploitation of the many results there are about harmonic functions for random walks on groups. However, the dependence on the inverse temperature $\beta$ presents issues for which there are no analogues in the random walk setting, for example the question about the behaviour of the equilibrium states as the temperature goes to zero which we study in detail in the present work. 

For abelian and more generally nilpotent groups there are results which describe the harmonic vectors of all non-negative matrices over the group that are consistent with the Cayley graph and where the passage to a stochastic matrix is therefore not necessary. This allows us to give a complete description of the KMS and $\mathrm{KMS}_{\infty}$ states for the generalized gauge actions with a potential function defined from a strictly positive function on the set of generators when the graph in question is the Cayley graph of a nilpotent group. The structure depends almost entirely on the abelianization of the group and from the methods we employ it becomes clear that a similar structure is present for arbitrary groups. To make this clear, and because it shows that our results have bearing for all finitely generated groups, we introduce \emph{abelian} KMS and $\mathrm{KMS}_{\infty}$ states in the general setting. They are the KMS and $\mathrm{KMS}_{\infty}$ states that arise from the harmonic vectors that factor through the abelianization of the group. We show that when the abelianization of the group is finite there is a unique abelian KMS state, and when the rank of the abelianization is $n \geq 1$ there is a critical inverse temperatur $\beta_0 > 0$ such that there are no abelian $\beta$-KMS states when $\beta < \beta_0$, a unique abelian $\beta_0$-KMS state and for $\beta > \beta_0$ the simplex of abelian $\beta$-KMS states is affinely homeomorphic to the Bauer simplex of Borel probability measures on the $(n-1)$-sphere. Based on this we determine the abelian $\mathrm{KMS}_{\infty}$ states; the ground states in the quantum statistical model that are limits of abelian $\beta$-KMS states when $\beta$ tends to infinity, \cite{CM}. The result is that they also form a Choquet simplex affinely homeomorphic to the Bauer simplex of Borel probability measures on the $(n-1)$-sphere. This may be expected given the description of the abelian $\beta$-KMS states, but one should bare in mind that it is a priori not even clear that the $\mathrm{KMS}_{\infty}$ states constitute a convex set and it is also not clear, in view of the massive collapse at the critical value $\beta_0$, that collapsing does not occur at infinity. The proof that the $(n-1)$-sphere 'survives to infinity' occupies almost half of the paper and involves a great deal of finite dimensional convex geometry. 

In a final section we consider two examples; the Heisenberg group and the infinite dihedral group. In both examples we consider a canonical set of generators and the gauge action on the resulting graph $C^*$-algebras. Since the Heisenberg group is nilpotent we easily find all KMS and $\mathrm{KMS}_{\infty}$ states by using our general results. For the infinite dihedral group, which is not nilpotent, we need to do a bit more work and find that there is only one abelian KMS state and a richer collection of general KMS states. This illustrates that for groups that are not nilpotent the structure of KMS states and $\mathrm{KMS}_{\infty}$ states is more complicated, and the abelian states will only give a (small) part of the picture. For general finitely generated groups, the complete picture is presently way out of reach.

While this work is the first to study KMS states and ground states for actions on the $C^*$-algebra of Cayley graphs, $\mathrm{KMS}_{\infty}$ states have been investigated in other cases, for example in \cite{LRR}, \cite{LR}, \cite{CL}, following their introduction by Connes and Marcolli in \cite{CM}. In many cases the set of all ground states as they are usually defined, e.g. in \cite{BR}, constitute a much larger set. This is also the case in our setting, where the set of ground states can be identified with the state space of a sub-quotient of the algebra, cf. \cite{Th2}.

\section{Generalized gauge actions on a pointed Cayley graph }
 Given a group $G$ and a finite set $Y$ of generators of $G$ there is a natural way of defining a directed graph $\Gamma = \Gamma (G,Y)$ whose vertexes are the elements of $G$ and with an edge (or arrow) from $g \in G$ to $h \in G$ iff $g^{-1}h \in Y$. This is the \emph{Cayley graph} and it provides the main tool for the geometric study of discrete finitely generated groups. Since the introduction of $C^*$-algebras from directed graphs, \cite{KPRR}, the Cayley graphs have provided a way of associating to a finitely generated discrete group a $C^*$-algebra $C^*(\Gamma)$ very different from the full or reduced group $C^*$-algebra usually considered in relation to discrete groups. The algebra is the universal $C^*$-algebra
generated by a set 
$V(g,s), \ g \in G, \ s \in Y$, of partial isometries such that
\begin{equation}\label{CKrel2}
V(h,t)^*V(g,s) = \begin{cases} 0 \ & \ \text{when} \ h \neq g \ \text{or} \ t \neq s, \\
\sum_{y \in Y} V(gy,y)V(gy,y)^* \ & \text{when} \ h = g \ \text{and} \ t = s .
\end{cases}
\end{equation}  
For $g \in G$ we let $P_g$ denote the projection
$$
P_g = \sum_{y \in Y} V(gy,y)V(gy,y)^* .
$$
When $G$ is infinite the graph $\Gamma$ is also infinite and the $C^*$-algebra $C^*(\Gamma)$ is not unital. But the neutral element $e_0$ of $G$ defines the canonical unital corner
$ P_{e_0}C^*(\Gamma)P_{e_0}$
stably isomorphic to $C^*(\Gamma)$, and in this paper we will focus attention to this corner of $C^*(\Gamma)$. 

We emphasize that the only condition on $Y$ is that it generates $G$ as a semi-group, i.e. every element of $G$ is a product of elements from $Y$. For various reasons it is convenient to exclude the case where $Y$ only contains one element. $G$ is a finite cyclic group when this happens, a case we do not exclude when $Y$ contains at least two elements. Thus we make the following standing assumption:

\begin{assumption} It is assumed that $Y \subseteq G$ is a finite set containing at least two elements and that it generates $G$ as a semi-group. 
\end{assumption}

Under this assumption $C^*(\Gamma)$ and the corner $P_{e_0}C^*(\Gamma)P_{e_0}$ are simple $C^*$-algebras by Corollary 6.8 of \cite{KPRR}.

Let $F : Y \to \mathbb R$ be a function. The universal property of $C^*(\Gamma)$ guarantees the existence of a continuous one-parameter group of automorphisms $\gamma^F_t, t \in \mathbb R$, on $C^*(\Gamma)$ defined such that 
$$
\gamma^F_t(V(g,s)) = e^{i F(s)t} V(g,s).
$$
Note that $\gamma^F$ keeps the corner $ P_{e_0}C^*(\Gamma)P_{e_0}$ globally invariant and therefore defines a continuous one-parameter group of automorphisms, still denoted by $\gamma^F_t$, on $  P_{e_0}C^*(\Gamma)P_{e_0}$.

We aim now to identify $P_{e_0}C^*(\Gamma)P_{e_0}$ as a $C^*$-subalgebra of the Cuntz-algebra $O_n$ where $n = \# Y$, cf. \cite{Cu}. To simplify notation, set
$$
O_Y = O_{\# Y}.
$$
Let $R_g, g \in G$, be the right-regular representation of $G$ on
$l^2(G)$ and $V_s, s \in Y$, the canonical isometries generating $O_Y$. Let $1_g \in B(l^2(G))$ be the orthogonal projection onto the subspace of $l^2(G)$ spanned by the characteristic function of $g \in G$. The elements
$$
V(g,s) = 1_{gs^{-1}} R_s \otimes V_s
$$
in $B(l^2(G)) \otimes O_Y$ satisfy the relations (\ref{CKrel2}) and hence they generate a copy of $C^*(\Gamma)$. For $t = (t_1,t_2, \cdots, t_n) \in Y^n$, set
$\overline{t} = t_1t_2\cdots t_n \in G$,
and let $V_{t} \in O_Y$ be the isometry
$$
V_{t} = V_{t_1}V_{t_2} \cdots V_{t_n} .
$$
Set $Y^0 = \emptyset$ and $\overline{\emptyset} = e_0, \ V_{\emptyset} = 1 \in O_Y$. The elements
\begin{equation}\label{form}
 1_h R_{\overline{t} \ \overline{u}^{-1}} \otimes V_{t}V_{u}^* ,
\end{equation}
where $t \in Y^n, \ u \in Y^m, \ n,m \in \mathbb N \cup \{0\}, \  h \in G$, span a $*$-subalgebra of $B(l^2(G)) \otimes O_Y$, and since
$$
(1_{e_0} \otimes 1) (1_h R_{\overline{t} \ \overline{u}^{-1}} \otimes V_{t}V_{u}^*)(1_{e_0} \otimes 1) = \begin{cases} 1_{e_0} \otimes  V_{t}V_{u}^* & \ \text{when} \ h = e_0 \ \text{and} \ \overline{t} = \overline{u} \\ 0, \ & \ \text{otherwise}, \end{cases}
$$
it follows that the elements $V_{t}V_{u}^*$ with $\overline{t} = \overline{u}$ span a $*$-subalgebra of $O_Y$ whose closure, which we denote by $O_Y(G)$, is a copy of $P_{e_0}C^*(\Gamma)P_{e_0}$.

To formulate what the action $\gamma^F$ looks like in this picture, set $F(\emptyset) = 0$ and
$$
F(w) = \sum_{j=1}^n F(w_j) 
$$
when $w = (w_1,w_2,\cdots, w_n) \in Y^n$. It follows from the universal property of $O_Y$ that there is a one-parameter group $\alpha^F$ of automorphisms on $O_Y$ such that
$$
\alpha^F_t\left(V_wV_u^*\right) =  e^{it\left(F(w)-F(u)\right)} V_wV_u^*
$$
for all $w,u \in \bigcup_{n=0}^{\infty} Y^n$. Note that $\alpha^F$ leaves $O_Y(G)$ globally invariant and defines a one-parameter group of automorphisms on $O_Y(G)$.

We summarise the preceding considerations with the following

\begin{prop}\label{0103} Let $O_Y(G)$ be the closed span in $O_Y$ of the elements $V_wV_u^*$ with $\overline{w} = \overline{u}$. There is a $*$-isomorphism $\pi : P_{e_0}C^*(\Gamma)P_{e_0} \to O_Y(G)$ such that $\pi \circ \gamma^F_t = \alpha^F_t \circ \pi$ for all $t \in \mathbb R$.
\end{prop}

\section{KMS measures and harmonic vectors} 

\subsection{KMS states and KMS measures}

Let $\beta \in \mathbb R$. A state $\omega$ on $O_Y(G)$ is a \emph{$\beta$-KMS state} for $\alpha^F$ when there is a dense $\alpha^F$-invariant $*$-subalgebra $\mathcal A$ of $O_Y(G)$ consisting of analytic elements for $\alpha^F$ such that
\begin{equation}\label{KMS}
\omega(ab) = \omega\left(b\alpha^F_{i\beta}(a)\right)
\end{equation}
for all $a,b \in \mathcal A$, cf. \cite{BR}. By Proposition 5.3.7 in \cite{BR} this condition is independent of $\mathcal A$, and since the elements $V_{t}V_{u}^*$ with $\overline{t} = \overline{u}$ span a $*$-subalgebra of $O_Y(G)$ consisting of analytic elements for $\alpha^F$ it follows that $\omega$ is a $\beta$-KMS state if and only if
\begin{equation}\label{KMS2}
\omega\left(V_{t_1}V_{u_1}^* V_{t_2} V_{u_2}^*\right) = e^{\beta\left(F(u_1) - F(t_1)\right)}\omega\left(V_{t_2}V_{u_2}^* V_{t_1} V_{u_1}^*\right)
\end{equation}
when $t_1,t_2,u_1,u_2 \in \bigcup_{n=0}^{\infty} Y^n$ and $\overline{t_i} = \overline{u_i}, \ i = 1,2$.

 It is well-known that the elements $V_tV_t^*, t \in \bigcup_{n=0}^{\infty} Y^n,$ generate a copy of $C\left(Y^{\mathbb N}\right)$ inside $O_Y(G)$ and that there is a conditional expectation $E : O_Y \to C\left(Y^{\mathbb N}\right)$ with the property that 
\begin{equation}\label{Pdense}
E\left(V_tV_u^*\right)  = \begin{cases} V_tV_t^* & \ \text{when} \ t = u \\ 0 & \ \text{otherwise.} \end{cases} 
\end{equation}

\begin{lemma}\label{prop1} Let $\beta \in \mathbb R$ and let $\omega$ be a $\beta$-KMS state for $\alpha^F$ on $O_Y(G)$. There is a unique Borel probability measure $m$ on $Y^{\mathbb N}$ such that
\begin{equation}\label{formula}
\omega(a) = \int_{Y^{\mathbb N}} E(a) \ \mathrm{d} m 
\end{equation}
for all $a \in O_Y(G)$.
\end{lemma}
\begin{proof} The conclusion can be obtained by combining Proposition 5.6 in \cite{CT} with Theorem 2.2 in \cite{Th1}, but in the present setting there is a much shorter proof: Let $l \in\bigcup_{n=1}^{\infty} Y^n$ be an element such that $\overline{l} = e_0$. Since $Y$ contains more than one element by assumption, there is an element $l' \in \bigcup_{n=1}^{\infty} Y^n$ with $\overline{l'} = e_0$ such that $l$ and $l'$ have different first entries. Then $V_lV_l^* + V_{l'}V_{l'}^* \leq 1$ and hence $\omega\left(V_lV_l^*\right) + \omega\left( V_{l'}V_{l'}^*\right) \leq 1$. The KMS condition \eqref{KMS2} shows that
$$
\omega\left( V_{l'}V_{l'}^*\right) = \omega\left( V_{l'}^*\alpha^F_{i\beta}\left(V_{l'}\right)\right) = e^{-\beta F(l')}\omega\left( V_{l'}^*V_{l'}\right) = e^{-\beta F(l')} > 0,
$$
implying that $e^{-\beta F(l)} = \omega\left(V_lV_l^*\right) < 1$.   
Consider then general elements $t,u \in \bigcup_{n=0}^{\infty} Y^n$ with $\overline{t} = \overline{u}$ such that $t \neq u$. It follows from the KMS-condition \eqref{KMS2} that
\begin{equation*}\label{25-04a}
\omega(V_{t}V_{u}^{*})= \omega(V_{t}V_{t}^{*}V_{t}V_{u}^{*}) =\omega(V_{t}V_{u}^{*}V_tV_t^*) .
\end{equation*}
In particular, $\omega(V_{t}V_{u}^{*}) =0$ if $V_u^*V_t = 0$. If $V_u^*V_t$ is not zero, there is an element $l \in\bigcup_{n=1}^{\infty} Y^n$ with $\overline{l} = e_0$ such that $u = tl$ or $t = ul$. Then
$$
\omega(V_{t}V_{u}^{*}) = \omega\left(\alpha^F_{i\beta} \left(V_{t}V_{u}^{*} \right)\right) = e^{\pm \beta F(l)} \omega(V_{t}V_{u}^{*}),
$$
where the sign depends on which of the two cases we are in. From the above we know that $e^{\beta F(l)} \neq 1$ and conclude therefore that $\omega(V_{t}V_{u}^{*}) = 0$ when $t \neq u$. This shows that $\omega = \omega \circ E$ and the statement in the lemma follows then from the Riesz representation theorem.
\end{proof}


\begin{defn}\label{def!} A Borel probability measure $m$ on $Y^{\mathbb N}$ is a \emph{$\beta$-KMS measure} for $\alpha^F$ when the state $\omega$ defined by (\ref{formula}) is a $\beta$-KMS state for $\alpha^F$.
\end{defn}

When $t = (t_1,t_2, \cdots, t_n) \in Y^n$ we denote in the following by $tY^{\mathbb N}$ the cylinder set
$$
tY^{\mathbb N} = \left\{(y_i)_{i=1}^{\infty} \in Y^{\mathbb N}: \ y_i = t_i, \ i = 1,2,\cdots, n\right\} ,
$$
and we set $\emptyset Y^{\mathbb N} = Y^{\mathbb N}$.

\begin{lemma}\label{betaKMS} A Borel probability measure $m$ on $Y^{\mathbb N}$ is a $\beta$-KMS measure if and only if
\begin{equation}\label{formula2}
e^{\beta F(t)}m\left(tY^{\mathbb N} \right) = e^{\beta F(u)}m\left(uY^{\mathbb N} \right) 
\end{equation}
whenever $t,u \in \bigcup_{n=0}^{\infty} Y^n$ satisfy that $\overline{t} = \overline{u}$.
\end{lemma}
\begin{proof} Let $\omega$ be the state of $O_Y(G)$ defined by (\ref{formula}). Let $t_1,u_1,t_2,u_2 \in \bigcup_n Y^n$ such that $\overline{t_i} = \overline{u_i}, i = 1,2$. Using the relations satisfied by the isometries $V_s$ and (\ref{Pdense}) we find 
$$
\omega\left(V_{t_1}V_{u_1}^* V_{t_2} V_{u_2}^*\right) = \begin{cases} \omega \left(V_{u_2}V_{u_2}^*\right) &  \text{if} \ t_2 =u_1x \ \text{and} \ u_2 = t_1x \ \text{for some} \ x \in \bigcup_{n=0}^{\infty} Y^n, \\ 
\omega\left(V_{t_1} V_{t_1}^*\right) & \ \text{if}  \ u_1 = t_2x  \ \text{and} \  t_1 = u_2x  \ \text{for some} \ x \in \bigcup_{n=0}^{\infty} Y^n, \\ 0 & \ \text{in all other cases} , \end{cases}
$$
while
\begin{equation*}
\begin{split}
&\omega\left(V_{t_2}V_{u_2}^* \alpha^F_{i\beta}\left(V_{t_1} V_{u_1}^*\right)\right) = e^{\beta\left(F(u_1) - F(t_1)\right)}\omega\left(V_{t_2}V_{u_2}^* V_{t_1} V_{u_1}^*\right)\\
& \\
& = \begin{cases} e^{\beta\left(F(u_1) - F(t_1)\right)} \omega \left(V_{u_1}V_{u_1}^*\right) &  \text{if} \  t_1  = u_2 x   \  \text{and} \ u_1 = t_2 x  \ \text{for some} \ x \in \bigcup_{n=0}^{\infty} Y^n, \\ 
e^{\beta\left(F(u_1) - F(t_1)\right)}\omega\left(V_{t_2} V_{t_2}^*\right) & \ \text{if}  \ u_2 = t_1 x \ \text{and} \ t_2 = u_1x,   \ \text{for some} \ x \in \bigcup_{n=0}^{\infty} Y^n, \\ 0 & \ \text{in all other cases} . \end{cases}
\end{split}
\end{equation*}
The two expressions agree for all choices of $t_1,u_1,t_2$ and $u_2$ if and only if (\ref{formula2}) holds for all $t,u \in \bigcup_{n=0}^{\infty} Y^n$ with $\overline{t} = \overline{u}$. Since the elements $V_{t}V_{u}^*$ are analytic elements for $\alpha^F$ it follows that $\omega$ is a $\beta$-KMS state for $\alpha^F$ if and only if (\ref{formula2}) holds for all $t,u \in \bigcup_{n=0}^{\infty} Y^n$ with $\overline{t} = \overline{u}$.
\end{proof}

\begin{cor}\label{KMS} The formula (\ref{formula}) gives an affine homeomorphism between the $\beta$-KMS measures on $Y^{\mathbb N}$ and the $\beta$-KMS states for $\alpha^F$.
\end{cor}

In terms of the canonical generators from Proposition \ref{0103} the $\beta$-KMS state $\omega$ corresponding to a $\beta$-KMS measure $m$ is given by the formula
$$
\omega\left(V_tV_u^*\right) = \begin{cases} m\left(tY^{\mathbb N}\right) & \ \text{when} \ t = u \\ 0 & \ \text{otherwise.} \end{cases} 
$$

\subsection{KMS measures and harmonic vectors}

A vector (or function) $\psi : G \to [0,\infty)$ will be called \emph{$\beta$-harmonic} when
\begin{equation}\label{harmeq}
\sum_{s \in Y} e^{-\beta F(s)}\psi_{gs} = \psi_g
\end{equation}
for all $g \in G$, and \emph{normalized} when $\psi_{e_0} = 1$.

\begin{lemma}\label{normharm} Let $\psi$ be a normalized $\beta$-harmonic vector. There is a unique $\beta$-KMS measure $m$ on $Y^{\mathbb N}$ such that
\begin{equation}\label{nhform}
m\left(tY^{\mathbb N}\right) = e^{-\beta F(t)} \psi_{\overline{t}}
\end{equation}
for all $t \in \bigcup_{n=0}^{\infty} Y^n$.
\end{lemma}
\begin{proof} It is standard to construct from $\psi$ a Borel probability measure on $Y^{\mathbb N}$ such that (\ref{nhform}) holds. For example one can apply Theorem 1.12 in \cite{Wo} with the stochastic matrix $p$ over $G$ defined such that
$$
p(g,h) = \psi_{g}^{-1} e^{-\beta F(g^{-1}h)}\psi_h
$$
when $g^{-1}h \in Y$ and $p(g,h) = 0$ otherwise, and with the initial distribution given by the Dirac measure supported on $e_0$. The resulting measure is clearly unique and it is a $\beta$-KMS measure by Lemma \ref{betaKMS}.
\end{proof}

As a converse to Lemma \ref{normharm}, note that a $\beta$-KMS measure $m$ defines a vector $\psi : G \to [0,\infty)$ such that
\begin{equation}\label{corrharm}
\psi_g = e^{\beta F(t)}m\left(tY^{\mathbb N}\right)
\end{equation}
for any choice of $t \in \bigcup_{n=0}^{\infty} Y^n$ with $\overline{t} = g$, and it is straightforward to check that $\psi$ is $\beta$-harmonic. Therefore
\begin{prop}\label{bac} The formula (\ref{nhform}) establishes a bijection between the $\beta$-KMS measures on $Y^{\mathbb N}$ and the normalized $\beta$-harmonic vectors.
\end{prop}

In conclusion, there are affine homeomorphisms between
\begin{enumerate}
\item[$\bullet$] the $\beta$-KMS states for $\alpha^F$,
\item[$\bullet$] the $\beta$-KMS measures on $Y^{\mathbb N}$, and
\item[$\bullet$] the normalized $\beta$-harmonic vectors on $G$, 
\end{enumerate}
given by (\ref{formula}) and (\ref{nhform}), respectively. As a consequence there are bijections between the extremal $\beta$-KMS states, the extremal $\beta$-KMS measures and the extremal normalized $\beta$-harmonic vectors.

\section{The abelian KMS states}\label{AbelKMS}

We say that a normalized $\beta$-harmonic vector $\psi$ is \emph{abelian} when 
$$
\psi_{hgk} = \psi_{hkg}
$$
for all $h,g,k \in G$. The abelian elements constitute a closed convex subset in the set of normalized $\beta$-harmonic vectors and we denote this set by $\Delta$.  A KMS state for $\alpha^F$ is \emph{abelian} when the associated normalized $\beta$-harmonic vector is abelian. Before we proceed with an investigation of abelian KMS states we point out that all KMS states are abelian when $G$ is nilpotent. This follows from the Krein-Milman theorem and the following result of Margulis, cf. \cite{Ma}, first proved for abelian groups by Doob, Snell and Williamson, \cite{DSW}.

\begin{thm}\label{margulis} (Margulis) Assume that $G$ is nilpotent. A normalized extremal $\beta$-harmonic vector $\psi : G \to [0,\infty)$ is multiplicative: $\psi_{gh} = \psi_g\psi_h$ for all $g,h \in G$.
\end{thm}

\begin{lemma}\label{??} An element $\psi \in \Delta$ is extremal in $\Delta$ if and only if there is there is a homomorphism $c : G \to \mathbb R$ such that 
\begin{equation}\label{c}
\psi_g = e^{c(g)} \ \ \ \forall g \in G .
\end{equation}
\end{lemma}
\begin{proof} Assume first that $\psi$ is extremal in $\Delta$. Since $\psi$ is abelian and $\beta$-harmonic,
\begin{equation}\label{9-1-17}
\psi_g = \sum_{s \in Y} e^{-\beta F(s)} \psi_{sg} = \sum_{s \in Y}  e^{-\beta F(s)} \psi_s \psi^s_g,
\end{equation}
where $\psi^s : G \to [0,\infty)$ is defined by
$\psi^s_g = \psi_s^{-1} \psi_{sg}$. Note that $\sum_{s \in Y}  e^{-\beta F(s)} \psi_s = \psi_{e_0} = 1$ and that $\psi^s$ is a normalized $\beta$-harmonic vector. Furthermore, $\psi^s$ is abelian since $\psi$ is. Since $\psi$ is extremal by assumption it follows therefore from (\ref{9-1-17}) that $\psi^s = \psi$ for all $s \in Y$. That is, $\psi_{sg} = \psi_s \psi_g$ for all $s \in Y$ and all $g \in G$. Since $Y$ generates $G$ as a semigroup it follows that $\psi_{hg} = \psi_h\psi_g$ for all $h,g \in G$. Set $c(g) = \log \psi_g$. 

 Conversely assume that there is a homomorphism $c : G \to \mathbb R$ such that (\ref{c}) holds. Consider an element $\phi \in \Delta$ and a $t > 0$ such that $t\phi_g \leq \psi_g$ for all $g \in G$. We must show that $\phi = \psi$. To this end we use Choquet theory to write
$$
\phi_g = \int_{\partial \Delta} \xi_g \ \mathrm{d}\nu(\xi) \ 
$$
where $\nu$ is a Borel probability measure on the set $ \partial \Delta$ of extreme points in $\Delta$, cf. e.g. Proposition 4.1.3 and Theorem 4.1.11 in \cite{BR}. When $\xi \in \partial \Delta \backslash \{\psi\}$
there is a $k \in G$ such that $\xi_k/\psi_k > 1$,
and hence also an open neighborhood $U$ of $\xi$ in $\partial \Delta$ such that
$\xi'_k/\psi_k > 1$ for all $\xi' \in U$. From the first part of the proof we know that the elements of $\partial \Delta$ are multiplicative and the same is true for $\psi $ by assumption. Hence
$$
\lim_{n \to \infty} \xi'_{nk} \left(\psi_{nk}\right)^{-1} = \lim_{n \to \infty}
\left(\xi'_k \left(\psi_k\right)^{-1}\right)^n = \infty
$$
for all $\xi' \in U$. But $t\phi \leq \psi$ by assumption,
so we must have that
$$
\int_{U } \xi'_{nk} \left(\psi_{nk}\right)^{-1} \ \mathrm{d}\nu(\xi') \leq \left(\psi_{nk}\right)^{-1}\int_{\partial \Delta} \xi'_{nk} \ \mathrm{d}\nu(\xi') = \left(\psi_{nk}\right)^{-1} \phi_{nk}\leq t^{-1}
$$
for all $n \in \mathbb N$ and we conclude therefore that $\nu(U) =
0$. Since $\xi\in  \partial \Delta \backslash \{\psi\}$ was
arbitrary it follows first that $\nu\left( \partial \Delta \backslash
  \{\psi\}\right) = 0$, and then that $\phi = \psi$.
\end{proof} 

A Borel measure $m$ on $Y^{\mathbb N}$ is \emph{abelian} when $m\left(t_1t_2t_3Y^{\mathbb N}\right) = m\left(t_1t_3t_2Y^{\mathbb N}\right)$ for all $t_1,t_2,t_3 \in \bigcup_n Y^n$. Clearly, a $\beta$-KMS measure is abelian if and only if the corresponding $\beta$-harmonic vector is.

 A Borel probability measure $m$ on $Y^{\mathbb N}$ is \emph{Bernoulli} when there is map, sometimes called a probability vector, $p : Y \to [0,1]$ with $\sum_{y \in Y} p(y)= 1$ such that $m$ is the corresponding  infinite product measure on $Y^{\mathbb N}$, i.e.
$$
m\left(tY^{\mathbb N}\right) =  \prod_{i=1}^n p(t_i)
$$
when $t = (t_i)_{i=1}^n \in Y^n$.

\begin{lemma}\label{bernoulli} An abelian $\beta$-KMS measure is extremal in the set of abelian $\beta$-KMS measures if and only if it is a Bernoulli measure.
\end{lemma} 
\begin{proof} Consider an abelian $\beta$-KMS measure $m$ and let $\psi$ be the corresponding $\beta$-harmonic vector. If $m$ is extremal in the set of abelian $\beta$-KMS measures $\psi$ is extremal in $\Delta$ and by Lemma \ref{??} there is a $c \in \Hom (G,\mathbb R)$ such that $\psi_g = e^{c(g)}$ for all $g$. The condition $\sum_{s \in Y} e^{-\beta F(s)}\psi_s = \psi_{e_0} = 1$ implies that 
$$
\sum_{s \in Y} e^{c(s) - \beta F(s)} = 1.
$$
Set $p(s) =  e^{c(s) - \beta F(s)}$.
For $t = (t_1,t_2,\cdots, t_n) \in Y^n$ we find that
$$
m\left(tY^{\mathbb N}\right) = \prod_{i=1}^n e^{-\beta F(t_i)}\psi_{t_1t_2\cdots t_n} = \prod_{i=1}^n e^{-\beta F(t_i) + c(t_i)} =\prod_{i=1}^n p(t_i) ,
$$
showing that $m$ is the Bernoulli measure on $Y^{\mathbb N}$ defined from $p$. This proves one of the implications, and to prove the reverse assume that $m$ is a $\beta$-KMS measure which happens to be Bernoulli. Let $\psi$ be the $\beta$-harmonic vector corresponding to $m$. Then $\psi$ is clearly abelian. When $t,u \in \bigcup_n Y^n$ we find from (\ref{nhform}) that
\begin{equation*}
\begin{split}
&\psi_{\overline{t}\overline{u}} = \psi_{\overline{tu}} = e^{\beta \left(F(t) + F(u)\right)}  m\left(tuY^{\mathbb N}\right) \\
&=  e^{\beta F(t)}  e^{\beta F(u)}  m\left(tY^{\mathbb N}\right) m\left(uY^{\mathbb N}\right) = \psi_{\overline{t}}\psi_{\overline{u}} .
\end{split}
\end{equation*}
Thus $g \to \psi_g$ is multiplicative and hence of the form (\ref{c}) for some $c \in \Hom(G,\mathbb R)$. It follows from Lemma \ref{??} that $\psi$ is extremal, and hence also that $m$ is.
\end{proof}

Set
$$
Q(\beta) = \left\{ c \in \Hom (G,\mathbb R): \ \sum_{s \in Y} e^{c(s)-\beta F(s)} = 1 \right\} .
$$
Equip $\mathbb R^G$ with the product topology and $Q(\beta) \subseteq \mathbb R^G$ with the relative topology. Then $Q(\beta)$ is a compact subset of $\mathbb R^G$. Given an element $c \in Q(\beta)$ we denote by $b_c$ the Bernoulli measure on $Y^{\mathbb N}$ defined from the probability vector $p(y) =  e^{c(y) - \beta F(y)}$.

\begin{lemma}\label{bernoulli2} The map $c \mapsto b_c$ is a homeomorphism from $Q(\beta)$ onto the set of extreme points of the abelian $\beta$-KMS measures equipped with the weak*-topology.
\end{lemma}
\begin{proof} $b_c$ is a $\beta$-KMS measure by Lemma \ref{normharm}. It follows from Lemma \ref{bernoulli} and its proof that $\left\{ b_c:  \ c \in Q(\beta)\right\}$ is the set of extreme points in the set of abelian $\beta$-KMS measures. Since the map $c \mapsto b_c$ is continuous it suffices to show that is it also injective; a fact which follows immediately from the observation that
$$
e^{c(s)} = e^{\beta F(s)} b_c\left(sY^{\mathbb N}\right) 
$$
for all $s \in Y$.
\end{proof}

\begin{thm}\label{Q-descr} There is an affine homeomorphism $\nu \mapsto \omega_{\nu}$ from the Borel probability measures $\nu$ on $Q(\beta)$ onto the abelian $\beta$-KMS states for $\alpha^F$ such that
\begin{equation}\label{bdint}
\omega_{\nu}(a) = \int_{Q(\beta)}  \int_{Y^{\mathbb N}} E(a) \  \mathrm{d} b_c \ \mathrm{d} \nu(c) 
\end{equation}
for $a \in O_Y(G)$.
\end{thm}
\begin{proof} Given a Borel probability measure $\nu$ on $Q(\beta)$ we can define a Borel probability measure $m$ on $Y^{\mathbb N}$ such that
$$
m(B) = \int_{Q(\beta)} b_c(B) \ \mathrm{d} \nu(c)
$$
for every Borel subset $B \subseteq Y^{\mathbb N}$. It is clear that $m$ is an abelian $\beta$-KMS measure since all $b_c, c \in Q(\beta)$ are, and it follows from Lemma \ref{bernoulli2} and Choquet theory that we obtain all abelian $\beta$-KMS measures $m$ on $Y^{\mathbb N}$ this way. Set
$$
\omega_{\nu}(a) = \int_{Y^{\mathbb N}} E(a) \ \mathrm{d} m,
$$
and note that \eqref{bdint} holds. Hence the map under consideration is surjective onto the abelian $\beta$-KMS states for $\alpha^F$. To see that it is also injective assume that $\nu$ and $\nu'$ are Borel probability measures on $Q(\beta)$ such that the corresponding states defined by (\ref{bdint}) are the same. Then
$$
\int_{Q(\beta)} \int_{Y^{\mathbb N}} f \ \mathrm{d} b_c  \ \mathrm{d} \nu(c) = \int_{Q(\beta)} \int_{Y^{\mathbb N}} f \ \mathrm{d} b_c  \ \mathrm{d} \nu'(c)
$$
for all $f \in C\left(Y^{\mathbb N}\right)$. Taking $f$ to be the characteristic function of $tY^{\mathbb N}$ we find that
$$
\int_{Q(\beta)} \prod_{i=1}^n e^{c(t_i)} \ \mathrm{d}\nu'(c) = \int_{Q(\beta)} \prod_{i=1}^n e^{c(t_i)} \ \mathrm{d}\nu(c)
$$
for all $t \in Y^n$ and all $n$. This shows that integration with respect to $\nu$ and $\nu'$ give the same functional on the algebra of functions on $Q(\beta)$ generated by the maps
$Q(\beta)  \ni c \mapsto e^{c(s)} , \ s \in Y$. This algebra is dense in $C(Q(\beta))$ by the Stone-Weierstrass theorem and it follows therefore that $\nu = \nu'$.

 \end{proof}

 \subsection{A closer look at $Q(\beta)$}

First the case where the abelianization of $G$ is finite:

\begin{prop}\label{Gabbb} When the abelianization $G/[G,G]$ of $G$ is trivial or a finite group there is an abelian $\beta$-KMS measure if and only if
\begin{equation}\label{cond}
\sum_{s \in Y} e^{-\beta F(s)} = 1.
\end{equation} 
When it exists, the abelian $\beta$-KMS measure is unique and it is the Bernoulli measure corresponding to the probability vector $p(s) = e^{-\beta F(s)}$.
\end{prop} 
\begin{proof} This follows straightforwardly from Theorem \ref{Q-descr} since $\Hom(G,\mathbb R) = \{0\}$ when $G/[G,G]$ is finite.
\end{proof}

\begin{cor}\label{Margulis1} Assume that $G$ is nilpotent and that the abelianization $G/[G,G]$ of $G$ is trivial or a finite group. There is a $\beta$-KMS measure if and only if
\begin{equation}\label{cond}
\sum_{s \in Y} e^{-\beta F(s)} = 1.
\end{equation} 
When it exists, the $\beta$-KMS measure is unique and it is the Bernoulli measure corresponding to the probability vector $p_s = e^{-\beta F(s)}$.
\end{cor} 

Note that the one-parameter group $\alpha^F$ is the restriction to $O_Y(G)$ of an action on $O_Y$. Any KMS state for the action on $O_Y$ will restrict to a KMS state for $\alpha^F$. It follows from work of Exel and Laca, \cite{EL}, at least when $F$ is strictly positive, that the action on $O_Y$ has exactly one KMS state. The abelian KMS state in Proposition \ref{Gabbb} is the restriction to $O_Y(G)$ of that state.

Consider now the case where the rank of $G/[G,G]$ is positive, say $n \geq 1$. Then $\Hom (G,\mathbb R) \simeq \mathbb R^n$ and we choose $n$ linearly independent elements $c'_i \in \Hom (G,\mathbb R), \ i = 1,2,\cdots, n$. For each $s \in Y$, set
$$
c_s = \left(c'_1(s),c'_2(s), \cdots, c'_n(s)\right) \in \mathbb R^n.
$$
Then
\begin{equation}\label{Qbetadef}
Q(\beta) \simeq \left\{ u\in \mathbb R^n: \ \sum_{s \in Y} \exp\left( u \cdot c_s - \beta F(s)\right) = 1 \right\} 
\end{equation}
when we let $\cdot$ denote the canonical inner product in $\mathbb R^n$.

\begin{lemma}\label{uniqueu} There is a unique vector $u(\beta) \in \mathbb R^n$ such that
\begin{equation}\label{ubeta2}
 \sum_{s \in Y}\exp\left(u\cdot c_s - \beta F(s)\right) \ > \  \sum_{s \in Y}\exp\left(u(\beta) \cdot c_s - \beta F(s)\right)
\end{equation}
 for all $u \in \mathbb R^n \backslash \{u(\beta)\}$. $u(\beta)$ is determined by the condition that
 \begin{equation}\label{ubeta}
\sum_{s \in Y} c_s \exp\left(u(\beta)\cdot c_s - \beta F(s)\right) = 0.
\end{equation}
\end{lemma}
\begin{proof} The function $\mathbb R^n \ni u \mapsto \sum_{s \in Y}e^{-\beta F(s)}
e^{u  \cdot c_s}$ is strictly convex since the exponential function
is. It has therefore at most one local minimum, which is necessarily a
global minimum. It follows that the global minimum, if it exists, occurs at the unique
$u(\beta) \in \mathbb R^n$ where the gradient is $0$, i.e. the vector $u(\beta)$ for which \eqref{ubeta} holds. It suffices therefore to show that
\begin{equation}\label{infty7} 
\lim_{\|u \| \to \infty} \sum_{s \in Y}\exp\left(u\cdot c_s - \beta F(s)\right) = \infty .
\end{equation}
To establish (\ref{infty7}) it suffices to show that for each $v \in \mathbb R^n, \|v\| =1$, there is an $s \in Y$ such that $v \cdot c_s > 0$. Assume for a contradiction that $\|v\| = 1$ and $v \cdot c_s \leq 0$ for all $s \in Y$. Define $c : G \to \mathbb R^n$ such that $c(g) = \left(c'_1(g),c'_2(g) , \cdots , c'_n(g)\right)$. Since $v \cdot c_s \leq 0$ it follows that $v \cdot c(g) \leq 0$ for all $g \in G$, and hence also $-v \cdot c(g) = v \cdot c\left(g^{-1}\right) \leq 0$ for all $g \in G$. This implies $v_1c'_1 + v_2c'_2 + \cdots + v_nc'_n = 0$, contradicting the choice of the $c'_i$'s.
\end{proof}

In view of Lemma \ref{uniqueu} and \eqref{Qbetadef} we must disinguish between the following three cases:
\smallskip

\subsubsection{ $\sum_{s \in Y}e^{-\beta F(s)} e^{u(\beta) \cdot c_s} > 1$.}\label{case1}

Then $Q(\beta) = \emptyset$.

\subsubsection{$\sum_{s \in Y}e^{-\beta F(s)} e^{u(\beta) \cdot c_s} = 1$.}\label{case2} Then $Q(\beta) = \{u_0\}$, where $u_0 \in \Hom(G,\mathbb R)$ is determined by the condition that
$u_0(s) = u(\beta) \cdot c_s \ \ s \in Y$. The corresponding $\beta$-KMS measure is the Bernoulli measure on $Y^{\mathbb N}$ defined by the probability vector $p_s =e^{-\beta F(s)} e^{u(\beta) \cdot c_s}$.

\subsubsection{$\sum_{s \in Y}e^{-\beta F(s)} e^{u(\beta) \cdot c_s} < 1$.}\label{case3}

Then $Q(\beta)$ is homeomorphic to the $(n-1)$-sphere
$$
S^{n-1} =  \left\{ v \in \mathbb R^{n} : \ \|v\| = 1 \right\}.
$$
This follows from

\begin{lemma}\label{what12}  Assume that $\sum_{s \in Y}e^{-\beta F(s)} e^{u(\beta) \cdot c_s} < 1$. For every $v \in S^{n-1}$ there is a unique positive number $t_{\beta}(v)$ such that
$$
u(\beta) + t_{\beta}(v)v \in Q(\beta)
$$
and the map $S^{n-1} \ni v \mapsto u(\beta) + t_{\beta}(v)v$ is a homeomorphism from $S^{n-1}$ onto $Q(\beta)$.
\end{lemma}

\begin{proof} 
Let $v \in \mathbb R^n$ be a unit
vector. The function $f_v : \mathbb R \to \mathbb R$ given by
$$
f_v(t) = \sum_{s \in Y} \exp \left((u(\beta)  +
  tv)\cdot c_s - \beta F(s)\right) = \sum_{s \in Y}e^{-\beta F(s)} e^{u(\beta) \cdot c_s}e^{t
  v\cdot c_s}
$$
has a unique local minimum when $t = 0$ where $f_v(0) < 1$. It follows from \eqref{infty7} that
$\lim_{t \to \pm \infty} f_v(t) = \infty$.
There are therefore unique real numbers $t_-,t_+ \in \mathbb R$ with $t_{-} < 0 < t_{+}$ such
that $f_v(t_-) = f_v(t_+) = 1$. Set $t_{\beta}(v) =t_+$ and note that $f_v'(t_{\beta}(v)) > 0$. It follows therefore from the
implicit function theorem that $t_{\beta}(v)$ is a differentiable, in particular continuous
function of $v$. As a consequence also the map 
$$
S^{n-1} \ni v \to u(\beta) + t_{\beta}(v)v \in Q(\beta)
$$
is continuous. It is easy to see that it is injective. To prove surjectivity let $u \in Q(\beta)$. Then $u \neq u(\beta)$ and we set
$$
v': = \frac{u - u(\beta)}{\lVert u - u(\beta) \rVert} \in S^{n-1}.
$$
Observe that $t_{\beta}(v') = \lVert u - u(\beta) \rVert$ since this is a positive number and
$$
\sum_{s \in Y}  \exp\left((u(\beta) + \lVert u - u(\beta) \rVert v') \cdot c_s - \beta F(s)\right) = \sum_{s \in Y} e^{-\beta F(s)} e^{u \cdot c_s}=1.
$$
It follows that $u(\beta)+t_{\beta}(v')v'=u$.
\end{proof}

The abelian $\beta$-KMS measure on $Y^{\mathbb N}$ corresponding to the vector $v \in S^{n-1}$ is the Bernoulli measure given by the probability vector $p$, where
$$
p(s)= \exp\left( \left(u(\beta) + t_{\beta}(v) v\right) \cdot c_s - \beta F(s)\right).
$$

\bigskip

\begin{assumption}
Assume now that $F(s) > 0$ for all $s \in Y$. 
\end{assumption}

\smallskip

\begin{lemma}\label{nl1} For each $u \in \mathbb{R}^{n}$ there exists a unique number $\beta(u) \in \mathbb{R}_{+}$ such that
$$
\sum_{s \in {Y}} e^{-\beta(u) F(s)} e^{u \cdot c_s}=1 \ .
$$
The function $\mathbb{R}^{n} \ni u \to \beta(u) \in \mathbb{R}_{+}$ is continuous.
\end{lemma}
\begin{proof}
Let $u \in \mathbb{R}^{n}$ and set
$$
g(t) = \sum_{s \in {Y}} e^{-t F(s)} e^{u \cdot c_s} .
$$
As shown in the proof of Lemma \ref{uniqueu} there must be some $s' \in {Y}$ such that $u \cdot c_{s'} \geq 0$. It follows that $g(0) > e^{u \cdot c_{s'}}\geq 1$. Since $F >0$, the function $g$ is strictly decreasing with limit $0$ at infinity so there is a unique number $\beta(u) \in ]0,\infty[$ such that $g(\beta(u)) = 1$. Continuity of the function $u \mapsto \beta(u)$ follows from the implicit function theorem.
\end{proof}

We shall need the following observation regarding the function $\beta(u)$:

\begin{lemma}\label{fremog}
Let $\{u_{m}\} \subseteq \mathbb{R}^{n}$. Then $\beta(u_{m}) \to \infty$ if and only if $\lVert u_{m} \rVert \to \infty$.
\end{lemma}
\begin{proof} Set $L = \max \left\{\left\|c_s\right\| : \ s \in Y \right\}$. Then
$$
1=\sum_{s \in {Y}} e^{-\beta(u_{m}) F(s)} e^{u_{m} \cdot c_s} \leq \sum_{s \in {Y}} e^{-\beta(u_{m}) F(s)} e^{\|u_m\|L} , 
$$
showing that $\beta(u_m) \to \infty \ \Rightarrow \ \|u_m\| \to \infty$.  Assume then that $\lVert u_{m} \rVert \to \infty$ and for a contradiction also that $\beta(u_{m}) \nrightarrow \infty$. After passage to a subsequence we can assume that $\beta(u_{m}) $ is bounded by some $K$ and that $u_{m} / \lVert u_{m} \rVert \to v$ for some $v \in S^{n-1}$. As shown in the proof of Lemma \ref{uniqueu} there is a $s' \in {Y}$ such that $v \cdot c_{s'} >0$. Hence we get that
$$
1  \geq e^{-K F(s')} \exp \left(\lVert u_{m} \rVert  \frac{u_{m}}{\lVert u_{m} \rVert }\cdot c_{s'} \right ) \to \infty,
$$
a contradiction.
\end{proof}

\begin{thm}\label{KMSmain} Assume that the abelianization $G/[G,G]$ has rank $n \geq 1$ and that $F(s) > 0$ for all $s \in Y$. It follows that there is a $\beta_0 > 0$ such that
\begin{enumerate}
\item[$\bullet$] there are no abelian $\beta$-KMS states for $\alpha^F$ when $\beta < \beta_0$,
\item[$\bullet$] there is a unique abelian $\beta_0$-KMS state for $\alpha^F$, and
\item[$\bullet$] for all $\beta > \beta_0$ the simplex of abelian $\beta$-KMS states for $\alpha^F$ is affinely homeomorphic to the simplex of Borel probability measures on the $(n-1)$ sphere $S^{n-1}$. 
\end{enumerate}
\end{thm}
\begin{proof} First observe that the function $\beta \to u(\beta)$ defined by Lemma \ref{uniqueu} is continuous. Indeed, assume that $\beta_n \to \beta$ in $\mathbb R$ and for a contradiction also that $u\left(\beta_n\right) \nrightarrow u(\beta)$. It follows from \eqref{infty7} that we can pass to a subsequence to arrange that $u(\beta_n) \to v \neq u(\beta)$. Then
$$
\lim_{n \to \infty} \sum_{s \in Y} e^{-\beta_n F(s)} e^{u(\beta_n) \cdot c_s} =  \sum_{s \in Y} e^{-\beta  F(s)} e^{v \cdot c_s} > \sum_{s \in Y} e^{-\beta  F(s)} e^{u(\beta) \cdot c_s},
$$
where the last inequality follows from \eqref{ubeta2}. It follows that for all large $n$,
$$
\sum_{s \in Y} e^{-\beta_n F(s)} e^{u(\beta_n) \cdot c_s} > \sum_{s \in Y} e^{-\beta_n F(s)} e^{u(\beta) \cdot c_s}, 
$$
in conflict with the definition of $u(\beta_n)$.

It follows from Lemma \ref{nl1} that there is $\beta \geq 0$ such that $\sum_{s \in Y} e^{-\beta F(s)} e^{u(\beta) \cdot c_s} \leq 1$. Set
$$
\beta_0 = \inf \left\{ \beta \in \mathbb R: \ \sum_{s \in Y} e^{-\beta F(s)} e^{u(\beta) \cdot c_s} \leq 1 \right\} .
$$
By continuity we must have that $\sum_{s \in Y} e^{-\beta_0 F(s)} e^{u(\beta_0) \cdot c_s} = 1$. Note that $\beta_0 > 0$ since $\sum_{s \in Y} e^{u(\beta_0) \cdot c_s} > 1$. For $\beta < \beta_0$ we are in case \ref{case1} and there are no $\beta$-KMS states since $Q(\beta) = \emptyset$. When $\beta = \beta_0$ we are in case \ref{case2} and there is a unique $\beta_0$-KMS state because $Q(\beta)$ contains exactly one element. Finally, when $\beta > \beta_0$, it follows from \eqref{ubeta2} that 
$$
\sum_{s \in Y} e^{-\beta F(s)} e^{u(\beta) \cdot c_s} \leq  \sum_{s \in Y} e^{-\beta F(s)} e^{u(\beta_0) \cdot c_s} < \sum_{s \in Y} e^{-\beta_0 F(s)} e^{u(\beta_0) \cdot c_s} =1.
$$
This means that we are in case \ref{case3} when $\beta > \beta_0$. 
\end{proof}

\begin{cor}\label{Margulis2}  Assume that $G$ is nilpotent, that the abelianization $G/[G,G]$ has rank $n \geq 1$ and that $F(s) > 0$ for all $s \in Y$. It follows that there is a $\beta_0 > 0$ such that
\begin{enumerate}
\item[$\bullet$] there are no $\beta$-KMS states for $\alpha^F$ when $\beta < \beta_0$,
\item[$\bullet$] there is a unique $\beta_0$-KMS state for $\alpha^F$, and
\item[$\bullet$] for all $\beta > \beta_0$ the simplex of $\beta$-KMS states for $\alpha^F$ is affinely homeomorphic to the simplex of Borel probability measures on the $(n-1)$ sphere $S^{n-1}$. 
\end{enumerate}
\end{cor}

\section{The abelian $\mathrm{KMS}_{\infty}$ states} 

Following \cite{CM} we say that a state $\omega$ on $O_Y(G)$ is a \emph{$\mathrm{KMS}_{\infty}$ state} when there is a sequence $\{\beta_n\} \subseteq \mathbb R$ and for each $n$ a $\beta_n$-KMS state $\omega_n$ such that $\lim_{n \to \infty} \beta_n = \infty$ and $\lim_{n \to \infty} \omega_n = \omega$ in the weak*-topology. When the $\omega_n$'s can be chosen as abelian $\beta_n$-KMS states we say that $\omega$ is an \emph{abelian $\mathrm{KMS}_{\infty}$ state}. It follows from Proposition \ref{Gabbb} that there are no abelian $\mathrm{KMS}_{\infty}$ states when the abelianization of $G$ is finite. We retain therefore here the assumption that the rank $n$ of $G/[G,G]$ is $\geq 1$. Furthermore, we assume also that $F$ is strictly positive and we denote by $\beta_0$ the least inverse temperature $\beta$ for which there are any abelian $\beta$-KMS states, cf. Theorem \ref{KMSmain}.


Let $\Delta_Y$ denote the simplex
$$
\Delta_Y = \left\{ p \in [0,1]^Y : \ \sum_{s \in Y} p_s = 1 \right\} .
$$
For $\beta > \beta_0$, set
$$
N_{\beta} = \left\{ \left( e^{-\beta F(s)} e^{u \cdot c_s}\right)_{s \in Y} : \ u \in Q(\beta) \right\} \subseteq \Delta_Y .
$$
Let $N_{\infty}$ denote the limit set of $N_{\beta}$ as $\beta \to \infty$; i.e. 
$$
N_{\infty} = \bigcap_{n \geq \beta_0} \  \overline{\bigcup_{\beta \geq n} N_{\beta}} .
$$

\begin{lemma} \label{l12}
Let $p \in N_{\infty}$. For all $\varepsilon >0$ there exists a $\beta_{\epsilon} >0$ such that for each $\beta \geq \beta_{\epsilon}$ there is a $x^{\beta} \in N_{\beta}$ with 
$\left|x^{\beta}_s \ -\ p_s\right|  \leq \varepsilon \  \ \forall s \in Y$.
\end{lemma}
\begin{proof} By assumption there are sequences $\{u_{n}\}$ and $\{\beta_n\}$ with $u_n \in Q(\beta_{n})$ such that $\beta_{n} \to \infty$ and $\lim_{n \to \infty} e^{-\beta_{n} F(s)}e^{u_{n} \cdot c_s}  =  p_s$ for all $s \in Y$. By choosing a subsequence, if necessary, we can assume that $\beta_{n} < \beta_{n+1}$ for all $n$. Set $\delta =\varepsilon / \# Y$ and choose $N \in \mathbb{N}$ such that
\begin{equation}  \label{e1}
\lvert e^{-\beta_{n}F(s)}e^{u_{n} \cdot c_s} - p_{s} \rvert \leq \delta \qquad \qquad \forall s \in {Y}
\end{equation}
when $n \geq N$. We claim that $\beta_{\epsilon} = \beta_N$ will do the job, so assume that $\beta \geq \beta_N$. There is an $n \geq N$ such that $\beta \in \left[\beta_n,\beta_{n+1}\right]$. Since the function $
[0,1] \ni \lambda \to \beta((1-\lambda) u_{n} +\lambda u_{n+1})$
is continuous and $\beta(u_n) = \beta_n, \ \beta(u_{n+1}) = \beta_{n+1}$  by Lemma \ref{nl1}, it follows that there is a $\lambda \in [0,1]$ such that $\beta = \beta \left( (1-\lambda)u_n + \lambda u_{n+1}\right)$. By convexity of the exponential function we have that
\begin{align*}
& \sum_{s \in Y} e^{-\left[(1-\lambda) \beta_n + \lambda \beta_{n+1}\right]} e^{ \left[(1-\lambda) u_n + \lambda u_{n+1}\right] \cdot c_s}  \\
& \leq (1 -\lambda)  \sum_{s \in Y} e^{- \beta_n} e^{u_n  \cdot c_s} + \lambda  \sum_{s \in Y} e^{- \beta_{n+1}} e^{u_{n+1}  \cdot c_s} = 1.
\end{align*}
It follows that
$$
\beta = \beta \left((1-\lambda)u_n + \lambda u_{n+1}\right) \leq (1-\lambda) \beta_n + \lambda \beta_{n+1} ,
$$  
and hence
\begin{equation*}
\begin{split}
& -\beta F(s) + \left((1-\lambda)u_n + \lambda u_{n+1}\right) \cdot c_s \\
& \geq  \left(1-\lambda\right)\left( - \beta_n F(s) + u_n \cdot c_s\right) +  \lambda \left( - \beta_{n+1} F(s) + u_{n+1} \cdot c_s\right)  \\
& \geq \min \left\{  - \beta_n F(s) + u_n \cdot c_s, \  - \beta_{n+1} F(s) + u_{n+1} \cdot c_s\right\} .
\end{split}
\end{equation*}
Therefore (\ref{e1}) implies that
$$
e^{-\beta F(s)}e^{((1-\lambda) u_{n} + \lambda u_{n+1})\cdot c_s} \ \geq \ p_s - \delta
$$
for all $s \in Y$. On the other hand, if there was a $\tilde{s} \in {Y}$ such that
$$
e^{-\beta F(\tilde{s})}e^{((1-\lambda) u_{n} + \lambda u_{n+1})\cdot c_{\tilde{s}}}  > p_{\tilde{s}}+\delta \left(\# Y\right),
$$
it would follow that
$$
1 = \sum_{s \in {Y}} e^{-\beta F(s)} e^{((1-\lambda) u_{n} + \lambda u_{n+1}) \cdot c_s} > p_{\tilde{s}}+\delta \left(\# Y\right) +\sum_{s \in {Y}\setminus\{\tilde{s}\}}(p_{s}-\delta) \geq \sum_{s \in {Y}}p_{s} = 1,
$$
which is absurd. Hence, for all $s \in {Y}$,
$$
p_{s}-\delta \leq e^{-\beta F(s)}e^{((1-\lambda) u_{n} + \lambda u_{n+1})\cdot c_s}  \leq p_{s}+\delta (\# Y) .
$$
Thus 
$$
x^{\beta} = \left( e^{-\beta F(s)}e^{((1-\lambda) u_{n} + \lambda u_{n+1})\cdot c_s}\right)_{s \in Y}
$$
 is an element of $N_{\beta}$ such that $\left|p_s - x^{\beta}_s\right| \leq \epsilon$ for all $s \in Y$.
\end{proof}

\begin{prop}\label{infty} Assume that $G/[G,G]$ is not finite and assume that $F$ is strictly positive. The abelian $\mathrm{KMS}_{\infty}$ states constitute a compact convex set affinely homeomorphic to the simplex of Borel probability measures on $N_{\infty}$. The abelian $\mathrm{KMS}_{\infty}$ state $\omega$ on $O_Y(G)$ corresponding to a Borel probability measure $\nu$ on $N_{\infty}$ is given by
\begin{equation}\label{infty2}
\omega(a) = \int_{N_{\infty}} \int_{Y^{\mathbb N}} E(a) \ \mathrm{d} n_p  \ \mathrm{d} \nu(p)
\end{equation}
for all $a \in O_Y(G)$, where $n_p$ is the Bernoulli measure defined by $p \in N_{\infty}$.  
\end{prop}
\begin{proof} Let $\beta > \beta_0$. Since the map $u \mapsto \left( e^{-\beta F(s)} e^{u \cdot c_s}\right)_{s \in Y}$ is a homeomorphism from $Q(\beta)$ onto $N_{\beta}$ it follows from Theorem \ref{Q-descr} that every abelian $\beta$-KMS state $\omega_{\beta}$ is given by a Borel probability measure $\nu$ on $N_{\beta}$ such that
\begin{equation*}\label{limit}
\omega_{\beta}(a)  =  \int_{N_{\beta}}\int_{Y^{\mathbb N}} E(a) \ \mathrm{d} n_p  \ \mathrm{d} \nu(p)
\end{equation*}
for all $a \in O_Y(G)$. Let $\{\omega_n\}$ be a sequence of $\beta_n$-KMS states such that $\lim_{n \to \infty} \beta_n = \infty$ and $\lim_{n \to \infty} \omega_n = \omega$ in the weak* topology.  Let $\nu_n$ be the Borel probability measure on $N_{\beta_n}$ corresponding to $\omega_n$. Extend $\nu_n$ to a Borel probability measure $\tilde{\nu_n}$ on $\Delta_Y$ such that
$$
\tilde{\nu_n}(B) = \nu_n\left(B \cap N_{\beta_n}\right)
$$
and let $\nu$ be a weak* condensation point of $\{\tilde{\nu_n}\}$ in the set of Borel probability measures on $\Delta_Y$. Then $\nu$ is concentrated on $N_{\infty}$ and \eqref{infty2} holds. This shows that the map from Borel probability measures on $N_{\infty}$ to states on $O_Y(G)$ given by \eqref{infty2} hits every abelian $\mathrm{KMS}_{\infty}$ state. The proof that the map is injective is identical with the proof of injectivity in Theorem \ref{Q-descr}.


It remains to show that for an arbitrary Borel probability measure $\nu$ on $N_{\infty}$ the state $\omega$ on $O_Y(G)$ defined by \eqref{infty2} is a $\mathrm{KMS}_{\infty}$ state. Since the set of $\mathrm{KMS}_{\infty}$ states is closed for the weak* topology it suffices to show this for a weak* dense subset of Borel probability measures on $N_{\infty}$, e.g. for the set of convex combinations of Dirac measures. For this set the claim follows straightforwardly from Lemma \ref{l12}.
\end{proof}

It remains to prove the following

\begin{thm}\label{Sn-1} The set $N_{\infty} \subseteq \Delta_Y$ is  homeomorphic to the $(n-1)$-sphere $S^{n-1}$.
\end{thm}

The proof of Theorem \ref{Sn-1} will occupy the next section, but we record here the following corollaries.

\begin{cor}\label{corMan} Assume that the abelianization $G/[G,G]$ has rank $n \geq 1$ and that $F(s) > 0$ for all $s \in Y$. The set of abelian $\mathrm{KMS}_{\infty}$ states for $\alpha^F$ on $O_Y(G)$ is a compact convex set affinely homeomorphic to the set of Borel probability  measures on the $(n-1)$-sphere.
\end{cor}

\begin{cor}\label{Main3} Let $G$ be a nilpotent group whose abelianization $G/[G,G]$ has rank $n \geq 1$ and assume that $F(s) > 0$ for all $s \in Y$. The set of $\mathrm{KMS}_{\infty}$ states for $\alpha^F$ on $O_Y(G)$ is a compact convex set affinely homeomorphic to the set of Borel probability measures on the $(n-1)$-sphere.
\end{cor}

\section{Proof of Theorem \ref{Sn-1}}

By definition $N_{\beta} \subseteq \Delta_{{Y}}$ is the image of the map
$$
Q(\beta) \ni u \to \Big( e^{-\beta F(s)}e^{u \cdot c_s} \Big)_{s \in {Y}} \in \Delta_{{Y}},
$$
and $N_{\infty}$ is the set of elements $t \in \Delta_{{Y}}$ for which there exist sequences $\{\beta_{m}\} \subseteq \mathbb{R}$ and $\{u_{m}\} \subseteq \mathbb{R}^{n}$ such that $u_{m} \in Q(\beta_{m})$ for all $m$, $\beta_{m} \to \infty$ and
$$
\Big( e^{-\beta_{m} F(s)}e^{u_{m} \cdot c_s} \Big)_{s \in {Y}} \to t \quad \text{ for } m \to \infty \ .
$$

\subsection{Partitioning of $S^{n-1}$ by polyhedral cones}
For any non-empty subset $ Z \subseteq {Y}$ we define $M(Z)$ to be the set
\begin{align*}
M(Z)=\Big\{v \in \mathbb{R}^{n} : \ v \cdot \left(\frac{c_s}{F(s)} - \frac{c_z}{F(z)}\right)& \leq 0 \quad \forall s \in {Y}\setminus Z , \ \forall z \in Z \ ,\\
&v \cdot \left(\frac{c_{z'}}{F(z')} - \frac{c_z}{F(z)}\right) = 0 \quad \forall z, z'\in  Z \Big\}.
\end{align*}
Notice that these sets are \emph{convex polyhedral cones}. In the following 'polyhedral cone' will always mean a cone of this form. We refer to \cite{Fu} and \cite{La} for the facts we need on such cones and which we state in the following. A \emph{face} of some $M(Z)$ is a subset $T$ of $M(Z)$ obtained by changing some of the inequalities in the definition into equalities, i.e. a face of $M(Z)$ is a set of the form $M(Z')$ with $Z \subseteq Z'$. Equivalently a convex subset $T$ is a face of $M(Z)$ if for any two distinct points $x,y \in M(Z)$ the implication $(x,y) \cap T \neq \emptyset \ \Rightarrow \ [x,y] \subseteq T$ holds, where $(x,y)$ is the open line segment from $x$ to $y$ and $[x,y]$ is the closed line segment from $x$ to $y$. Each face of $M(Z)$ is again a convex polyhedral cone, and a face of a face is a face. The intersection of two polyhedral cones $M(Z)$ and $M(Z')$ is again a polyhedral cone, equal to $M(Z \cup Z')$. In particular, for every polyhedral cone $M(Z)$ there is a unique subset $Z' \subseteq Y$, which we call \emph{maximal} with the property that $M(Z) = M(Z')$ and $M(Z) = M(Z'') \Rightarrow Z'' \subseteq Z'$. For all $Z \subseteq {Y}$ we let $\mathcal{F}(M(Z))$ be the set of faces in $M(Z)$, which is a finite set. A \emph{proper face} of $M(Z)$ is a face $T$ of $M(Z)$ with $T \neq M(Z)$; we denote the set of these by $\mathcal{F}_{0}(M(Z))$. We define the dimension of a polyhedral cone $M(Z)$ to be $\text{dim}(M(Z)-M(Z))$, i.e. the dimension of the smallest subspace of $\mathbb R^n$ containing $M(Z)$, and we call something a \emph{$k$-face} if it is a face of dimension $k$. A \emph{facet} of $M(Z)$ is then a face of $M(Z)$ of dimension $\text{dim}(M(Z))-1$. It is well-known that any proper face of a polyhedral cone $M(Z)$ is contained in a facet of $M(Z)$.

A polyhedral cone $M(Z)$ is \emph{strongly convex}, meaning that $M(Z) \cap (-M(Z))=\{0\}$. To see this, take a $v \in M(Z) \cap (-M(Z))$ and a $z\in Z$. Then $v$ satisfies:
$$
v \cdot\left( \frac{c_s}{F(s)}-\frac{c_z}{F(z)} \right) =0 \qquad \forall s \in {Y} 
$$ 
However if $v \cdot c_z \geq 0$ this would imply that $v \cdot c_s \geq 0$ for all $s \in {Y}$ which can not be true unless $v=0$, cf. the proof of Lemma \ref{uniqueu}, and likewise $v \cdot c_z \leq 0$ would imply that $v \cdot c_s \leq 0$ for all $s \in {Y}$ which also implies $v=0$.

When we use the expression $\mathrm{Int}(M(Z))$ we mean the topological interior of $M(Z)$ in the subspace $M(Z)-M(Z)$ with the relative topology. For each of our polyhedral cones $M(Z)$ of dimension at least $1$ we define an element $c(M(Z)) \in \mathrm{Int}(M(Z)) \cap S^{n-1}$ which we call the \emph{center} of $M(Z)$ as follows: Say $M(Z)$ has $q$ $1$-faces $T_{1}, \dots, T_{q} $. Since $M(Z)$ is strongly convex we can then write $M(Z) = \{r_{1}v_{1} +\dots + r_{q}v_{q} \ :  \ r_{i} \geq 0\}$ where each $v_{i} $ is the unique $ v_{i} \in T_{i} \cap S^{n-1}$, by $(13)$ of Section 1.2 in \cite{Fu}. We set
$$
c(M(Z)): = \frac{\frac{1}{q} \sum_{i=1}^{q} v_{i}}{\lVert \frac{1}{q} \sum_{i=1}^{q} v_{i} \rVert}
$$ 
and then $c(M(Z)) \in \mathrm{Int}(M(Z))$, cf. \cite{Fu}.
\begin{lemma} \label{l204}
Let $M(Z)$ be a polyhedral cone with $Z$ chosen maximal. The following holds: 
\begin{enumerate}
\item\label{1.1} If $T \in \mathcal{F}(M(Z))$ and $T=M(Z')$ with $Z'$ chosen maximal, then $Z \subseteq Z'$ and $M(Z) =T$ if and only if $Z=Z'$. 
\item\label{1.3} For each $z \in Z$,
\begin{equation*}
\begin{split}
& \mathrm{Int}(M(Z)) \\
& = \left\{ v \in \mathbb{R}^{n} : \ v \cdot\left( \frac{c_s}{F(s)} -\frac{c_z}{F(z)} \right) <0 \ \forall s\in {Y}\setminus Z   , \ v \cdot\left( \frac{c_s}{F(s)} -\frac{c_z}{F(z)} \right) =0  \ \forall s \in Z  \right\}\\
& =
M(Z) \setminus \left( \bigcup_{T \in \mathcal{F}_{0}(M(Z))} T  \right).
\end{split}
\end{equation*}
\end{enumerate}
\end{lemma}
\begin{proof} (1): Since $M(Z') \subseteq M(Z)$ we have that $M(Z')=M(Z') \cap M(Z)=M(Z \cup Z')$, so $Z \cup Z' \subseteq Z'$ by maximality, and hence $Z \subseteq Z'$. That $Z=Z' \iff M(Z) = M(Z')$ follows directly from the way we defined $M$ and the fact that $Z$ and $Z'$ both are maximal. (2): The second equality is $(7)$ in section $1.2$ of \cite{Fu} and the first follows from the maximality of $Z$. 
\end{proof}

Set $\mathrm{bd}(M(Z))=M(Z) \setminus \mathrm{Int}(M(Z))$.

\begin{lemma}\label{opsplit}
Let $M(Z)$ be at least $2$-dimensional. Fix $v \in \mathrm{Int}(M(Z)) \cap S^{n-1}$. For every $w \in (M(Z)\cap S^{n-1}) \setminus \{v\}$ there exists a unique pair $(\lambda, u)$ where $\lambda \in ]0,1]$ and $u \in \mathrm{bd}(M(Z))\cap S^{n-1}$ such that
$$
w = \frac{(1-\lambda)v+\lambda u}{\lVert(1-\lambda)v+\lambda u \rVert}
$$
\end{lemma}
\begin{proof} We may assume that $Z$ is maximal. Consider the subspace $W:=\text{span}\{v,w\}$. Note that $w \neq -v$ since $M(Z) \cap \left(-M(Z)\right) = \{0\}$. Hence $\text{dim}(W)=2$. Since $W \cap M(Z)$ is a closed convex cone in $W$ and $-v \notin W \cap M(Z)$, the circle arch in $W$ starting in $v$ and continuing through $w$ must reach the boundary of $M(Z) \cap W$ at some $u$ with an angle to $v$ less than $\pi$ and 
\begin{equation}\label{holds}
w=\frac{(1-\lambda)v+\lambda u}{\lVert (1-\lambda)v+\lambda u\rVert } .
\end{equation}
for some $\lambda \in ]0,1]$. Since $u$ lies in the boundary of $M(Z) \cap W$ it can be approximated by elements from $M(Z)^{c}\cap W$ and hence $u \in \mathrm{bd}(M(Z))$. To establish the uniqueness part, assume \eqref{holds} holds with $(\lambda,u)$ replaced by the pair $(\lambda', u') \in  ]0,1] \times \mathrm{bd}(M(Z))\cap S^{n-1}$. Fix a $z \in Z$. It follows from \eqref{1.1} in Lemma \ref{l204} that there is a subset $Z_1 \subseteq Y$ such that $u \in M(Z_{1})$ and $Z \subsetneq Z_{1}$. We can assume that
$$
\alpha =  \frac{(1-\lambda')}{\lVert(1-\lambda')v+\lambda' u' \rVert}-\frac{(1-\lambda)}{\lVert(1-\lambda)v+\lambda u \rVert} \geq 0 \ .
$$ 
Consider an element $s \in Z_{1} \setminus Z$ and set $q:=c_s/F(s)-c_z/F(z)$. Then 
$$
0=q \cdot \frac{\lambda u}{\lVert(1-\lambda)v+\lambda u \rVert} = q \cdot \alpha v +  q \cdot \frac{\lambda' u'}{\lVert(1-\lambda')v+\lambda' u' \rVert}.
$$ 
But $q \cdot  v < 0$ since $s \notin Z$ and  $q \cdot \lambda' u' \leq 0$ since $u' \in M(Z)$, and hence $\alpha =0$. This implies that
$$
u \cdot \frac{\lambda }{\lVert(1-\lambda)v+\lambda u \rVert}=u' \cdot \frac{\lambda' }{\lVert(1-\lambda')v+\lambda' u' \rVert}.
$$
Since $u, u' \in S^{n-1}$ and $\alpha = 0$ it follows first that  $\lVert(1-\lambda)v+\lambda u \rVert = \lVert(1-\lambda')v+\lambda' u' \rVert$ and $\lambda = \lambda'$, and then also that $u = u'$. 
\end{proof}

We denote the $u$ of Lemma \ref{opsplit} by $P(w)$ and the $\lambda$ by $\lambda_{w}$. We suppress the $v$ in the notation because it will always be $c(M(Z))$ in the following.

\begin{lemma}\label{CONT0}
The maps $w \mapsto P(w)$ and $w \mapsto \lambda_{w}$ are continuous from $(M(Z) \cap S^{n-1}) \setminus \{v\}$ to $S^{n-1}$ and $[0,1]$ respectively.
\end{lemma}
\begin{proof}
Assume that $\{w_{m}\} \subseteq (M(Z) \cap S^{n-1}) \setminus \{v\}$  converges to a $w$ in this set. If $\lambda_{w_{m}} \nrightarrow \lambda_{w}$ or $P(w_{m}) \nrightarrow P(w)$ then by compactness of $[0,1]$ and $S^{n-1}$ we can take a subsequence $\{w_{m_{i}}\}$ of $\{w_{m}\}$ such that $\lambda_{w_{m_{i}}} \to \lambda$ and $P(w_{m_{i}}) \to u$, with either $\lambda \neq \lambda_{w}$ or $u \neq P(w)$.  Since $\lim_{i} w_{m_{i}} = w$ we then have:
\begin{align*}
&\frac{(1-\lambda_{w})v+\lambda_{w}P(w)}{\lVert (1-\lambda_{w})v+\lambda_{w}P(w) \rVert}=w  = \lim_{i}  \frac{(1-\lambda_{w_{m_{i}}})v+\lambda_{w_{m_{i}}}P(w_{m_{i}})}{\lVert (1-\lambda_{w_{m_{i}}})v+\lambda_{w_{m_{i}}}P(w_{m_{i}}) \rVert} = \frac{(1-\lambda)v+\lambda u}{\lVert (1-\lambda)v+\lambda u \rVert}.
\end{align*}
The uniqueness part of Lemma \ref{opsplit} implies that $\lambda=\lambda_{w}$ and $u = P(w)$, giving us the desired contradiction.
\end{proof}

\subsection{Constructing a homeomorphism $H: S^{n-1} \to N_{\infty}$}

For each $l \in \{1 , \dots, n \}$ we define the \emph{$l$-skeleton} as the union of all sets $M(Z) \cap S^{n-1}$ with $M(Z)$ a $l'$-dimensional polyhedral cone for some $1 \leq l' \leq l$. The skeletons will be used to give a recursive definition of $H$, but we need some preparations for this.

\begin{lemma}\label{nskel} The $n$-skeleton is all of $S^{n-1}$.
\end{lemma}
\begin{proof} Let $v \in S^{n-1}$ and choose $s \in Y$ such that $v \cdot \frac{c_y}{F(y)} \leq v \cdot \frac{c_s}{F(s)}$ for all $y\in Y$. Then $v \in M(\{s\})$.
\end{proof}

\begin{defn}
We say a sequence $\{v_{k}\}_{k \in \mathbb{N}} \subseteq \mathbb{R}^{n}$ is \emph{associated} with a $t \in N_{\infty}$ when $\beta(v_{k}) \to \infty$ and
$$
\left(  e^{-\beta(v_{k})F(s)+v_{k} \cdot c_s} \right)_{s \in {Y}} \ \to \ t
$$
for $k \to \infty$.
\end{defn}


\begin{lemma}\label{niceface}
Assume $\{u_{k}\}_{k \in \mathbb{N}}$ is associated with $t \in N_{\infty}$ and that $u_{k}  \in M(Z)$ for all $n$. Then 
\begin{enumerate}
\item\label{2.1}  $ t_{s}^{1/F(s)} \leq t_{z}^{1/F(z)} \ \forall z \in Z\ \forall s \in Y$, and
\item\label{2.2} $t_{z} \neq 0$ and $t_{z}^{1/F(z)} = t_{y}^{1/F(y)}$ for all $z,y \in Z$.
\end{enumerate}
\end{lemma}
\begin{proof} Since $u_{k} \in M(Z)$ it follows that $ u_{k} \cdot \frac{c_{s}}{F({s})} \ \leq \ u_{k} \cdot \frac{c_z}{F(z)}$ and hence that 
\begin{align*}
-\beta(u_{k})F({s})+ u_{k} \cdot c_{s} \ \leq \ \frac{F({s})}{F(z)} \left( -\beta(u_{k})F(z)+ u_{k} \cdot c_z \right) 
\end{align*}
for all $k$ when $s \in Y$ and $z \in Z$. This shows that (1) holds. (2) follows from (1).
\end{proof}

\begin{lemma} \label{forhold}
Let $t,\tilde{t} \in \Delta_{{Y}}$. If there are $y , y' \in {Y}$ such that $t_{y} \neq 0$, $\tilde{t}_{y'} \neq 0$ and 
$$
\frac{t_{s}^{1/F(s)}}{t_{y}^{1/F(y)}}=\frac{\tilde{t}_{s}^{1/F(s)}}{\tilde{t}_{y'}^{1/F(y')}} \qquad \forall s \in {Y},
$$
then $t=\tilde{t}$.
\end{lemma}

\begin{proof} Note that $\tilde{t}_{s}=\tilde{t}_{y'}^{F(s)/F(y')} \left( t_{s}^{1/F(s)} /t_{y}^{1/F(y)}\right)^{F(s)}$ and hence 
\begin{equation}\label{0903}
\sum_{s \in {Y}} \left(t_{y}^{1/F(y)}\right)^{F(s)} \left( \frac{t_{s}^{1/F(s)}}{t_{y}^{1/F(y)}} \right)^{F(s)} = \sum_{s \in {Y}} t_{s} = 1 = \sum_{s \in {Y}} \tilde{t}_{s} = \sum_{s \in {Y}} \left(\tilde{t}_{y'}^{1/F(y')}\right)^{F(s)} \left( \frac{t_{s}^{1/F(s)}}{t_{y}^{1/F(y)}} \right)^{F(s)}.
\end{equation}
Since the function
$$
]0,\infty[ \ni x \mapsto \sum_{s \in {Y}} x^{F(s)} \left( \frac{t_{s}^{1/F(s)}}{t_{y}^{1/F(y)}} \right)^{F(s)}
$$
is strictly increasing it follows from \eqref{0903} that $t_{y}^{1/F(y)} = \tilde{t}_{y'}^{1/F(y')}$ which yields the conclusion.
\end{proof}

\begin{lemma}\label{l401}
Assume $M(Z)$ is at least 1-dimensional and that $ Z \subseteq {Y}$ is maximal. There is a unique element $t \in N_{\infty}$ satisfying the following two conditions:
\begin{enumerate}
\item\label{3.1} $t_{s} \neq 0$ if and only if $s \in Z$, and 
\item\label{3.2} $t_{s_{1}}^{1/F(s_{1})} = t_{s_{2}}^{1/F(s_{2})}$ for all $s_{1}, s_{2} \in Z$. 
\end{enumerate}
Furthermore, $\lim_{r \to \infty} e^{-\beta(r v)F(s)+rv \cdot c_s} = t_s$ for all $s \in Y$ when $v \in \mathrm{Int}(M(Z))$.
\end{lemma}
\begin{proof} Let $v \in \mathrm{Int}(M(Z))$. Then $v \neq 0$ and $\lim_{r \to \infty} \beta(r v) = \infty$ by Lemma \ref{fremog}. Since $\Delta_{{Y}}$ is compact there is a subsequence $\{r_{i} v \}_{i \in \mathbb{N}}$ and an element $t \in N_{\infty}$ such that $\{r_{i} v \}_{i \in \mathbb{N}}$ is associated with $t$. It follows as in the proof of Lemma \ref{l12} that $\lim_{r \to \infty} e^{-\beta(r v)F(s)+rv \cdot c_s} = t_s$ for all $s \in Y$. To prove (1) notice that $r_{i} v \in M(Z)$ and hence $t_{z} \neq 0$ for $z \in Z$ by Lemma \ref{niceface}. Assume then that $t_{s} \neq 0 $ for a $s \in {Y}$, and assume for contradiction that $s \notin Z$. Fix a $y \in Z$. It follows from \eqref{1.3} of Lemma \ref{l204} that
\begin{equation}\label{0903a}
\frac{v \cdot c_s}{F(s)} \ < \ \frac{v \cdot c_y}{F(y)} .
\end{equation}
Since $t_{s} \neq 0$ and $\lim_{r \to \infty} \left(-\beta(rv)F(s)+rv \cdot c_s\right) = \log(t_{s})$, it follows that $\lim_{r\to \infty} \beta(rv)/r=v \cdot c_s/F(s)$. It follows then from (\ref{0903a}) that is a $L>0$ such that $\beta(rv)/r < v \cdot c_y /F(y)$ for all $r >L$. But this means that
$$
e^{-\beta(rv)F(y) +rv \cdot c_y}>1,
$$ 
when $r > L$, contradicting the definition of $\beta(rv)$, cf. Lemma \ref{nl1}. Hence (1) holds, and (2) follows from Lemma \ref{niceface}. For uniqueness, consider an element $t' \in \Delta_{{Y}}$ for which (1) and (2) hold. Let $z \in Z$ and note that 
$$
1= \sum_{s \in {Y}} t'_{s} = \sum_{s \in Z} t'_{s} = \sum_{s \in Z} {t'}_{z}^{F(s)/F(z)} =\sum_{s \in Z} {t}_{z}^{F(s)/F(z)} .
$$
Since the function $ ]0,\infty[ \ni x \to \sum_{s \in Z} x^{F(s)/F(y)}$ is strictly increasing it follows that $t_z = t'_z$, and hence that $t = t'$.
\end{proof}

\subsubsection{$H$ on the $1$-skeleton}

Let $M(Z)$ be a 1-dimensional polyhedral cone with $Z$ chosen maximal. Then $M(Z) \cap S^{n-1}$ consists only of one point $x$. We define $H(x)=t$, where $t\in N_{\infty}$ is the unique element obtained from Lemma \ref{l401} using $M(Z)$.


\subsubsection{$H$ on the $k$-skeleton} $H$ will be defined inductively and the basic idea is illustrated by the picture below. Everything inside the figure represents the intersection between $S^{2}$ and some $3$-dimensional polyhedral cone $M(Z)$. The green dots are the $1$-dimensional faces intersected with $S^{2}$, the blue lines (containing the green dots) are the $2$-dimensional faces intersected with $S^{2}$. Assuming that we have defined $H$ on the $2$-skeleton, we have defined $H$ on the blue lines and the green dots. As a step to define $H$ on the $3$-skeleton we first define $H$ on $c(M(Z))$ which lies in the interior $\mathrm{Int}(M(Z))$, and hence is not in the $2$-skeleton. Then for any $v \neq c(M(Z))$, we define $H$ on $v$ depending on $P(v)$ and how close it lies to the center $c(M(Z))$. We measure the distance to $c(M(Z))$ by using the unique decomposition obtained in Lemma \ref{opsplit}.

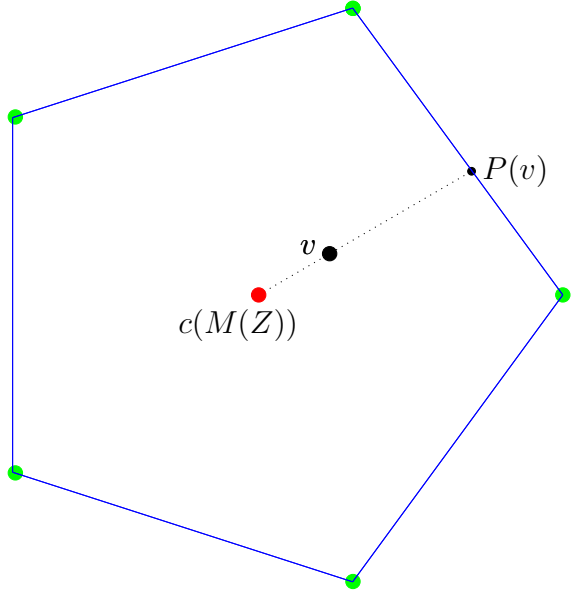
\begin{center}
\begin{tikzpicture}[scale=4]
\draw[dotted] (0,0) -- (0.7,0.41);
\draw[red,fill=red] (0,0) circle (.13ex);
\draw[green,fill=green] (1,0) circle (.13ex);
\draw[green,fill=green] (-0.8,0.59) circle (.13ex);
\draw[green,fill=green] (0.31,0.95) circle (.13ex);
\draw[green,fill=green] (0.31,0.95) circle (.13ex);
\draw[green,fill=green] (-0.8,-0.59) circle (.13ex);
\draw[green,fill=green] (0.31,-0.95) circle (.13ex);
\draw[black,fill=black] (0.2333,0.1367) circle (.13ex);
\draw[black,fill=black] (0.7,0.41) circle (.07ex);
  \foreach \x in {0,72,...,288} {
        \draw[fill, blue] (\x:1 cm) -- (\x + 72:1 cm);
        }
\node[text width=1cm, anchor=west, right] at (-0.3,-0.1) {$c(M(Z))$};
\node[text width=1cm, anchor=west, right] at (0.1,0.16) {$v$};
\node[text width=1cm, anchor=west, right] at (0.1,0.16) {$v$};
\node[text width=1cm, anchor=west, right] at (0.7,0.41) {$P(v)$};
\end{tikzpicture}
\end{center}


For the procedure to work we have to impose conditions at each step. We say that $H$ satisfies \emph{the induction conditions on the $l$-skeleton} if $H$ is defined on the $l$-skeleton and has the following properties.
\begin{enumerate}
\item\label{4.1} $H$ is continuous on the $l$-skeleton.
\item\label{4.2} If $x \in M(Z)\cap S^{n-1}$ with $\text{dim}(M(Z)) \leq l$, then there is a sequence $\{v_{m}\} \subseteq M(Z)$ associated with $H(x)$.
\item\label{4.3} For $x$ in the $l$-skeleton, let $y\in {Y}$ satisfy that $H(x)^{1/F(s)}_s \leq H(x)^{1/F(y)}_y \ \forall s \in Y$. Then $x \in \mathrm{Int}(M(Z))$ where 
$$
Z = \{s \in {Y} \ :  \ H(x)_{s}^{1/F(s)}=H(x)_{y}^{1/F(y)}\},
$$
and $Z$ is maximal for $M(Z)$.
\end{enumerate}

\begin{lemma}\label{obsH}
Assume $H$ satisfy the induction condition on the $l$-skeleton. Let $x \in M(Z)$ with $\text{dim}(M(Z))\leq l$ and $Z$ chosen maximal. Then $H(x)_{z}^{1/F(z)} \geq H(x)_{s}^{1/F(s)}$ for all $ s \in {Y} $ and $z \in Z$. In particular $H(x)_{z} \neq 0$ for $z \in Z$.
\end{lemma}
\begin{proof} Combine induction condition \eqref{4.2} with Lemma \ref{niceface}.
\end{proof}

Notice that $H$ satisfies the induction conditions on the $1$-skeleton. For the induction step, assume that we have defined $H$ on the $(k-1)$-skeleton for some $k>1$ and that $H$ satisfies the induction conditions on the $(k-1)$-skeleton.

Now we will define $H$ on the $k$-skeleton. So let $M(Z) $ be a polyhedral cone of dimension $k$ and let $Z$ be chosen maximal. An element $v \in \mathrm{Int}(M(Z))$ does not lie in the $k-1$-skeleton by \eqref{1.3} of Lemma \ref{l204}, while a $v \in \mathrm{bd}(M(Z)) \cap S^{n-1}$ does, so we want to define $H$ on $\mathrm{Int}(M(Z)) \cap S^{n-1}$. Since $M(Z)$ is a strongly convex polyhedral cone we can consider $c(M(Z)) \in \mathrm{Int}(M(Z)) \cap S^{n-1}$, and set $H(c(M(Z)))=t$ where $t$ is the unique element arising from Lemma \ref{l401} using $M(Z)$.

\begin{lemma}\label{kdef1403}
 For any $v \in \left(S^{n-1}\cap M(Z)\right) \backslash \{c(M(Z))\}$ write
$$
v = \frac{(1-\lambda) c(M(Z))+\lambda P(v)}{\lVert (1-\lambda) c(M(Z))+\lambda P(v) \rVert},
$$
where $\lambda \in ]0,1]$ and $P(v) \in \mathrm{bd}(M(Z))$ are unique, cf. Lemma \ref{opsplit}. There is a unique $t \in N_{\infty}$ such that
\begin{equation}\label{33}
t_{s}^{1/F(s)} = t_{z}^{1/F(z)} \frac{H(P(v))_{s}^{1/F(s)}}{H(P(v))_{z}^{1/F(z)}} \exp\left(-\log(\lambda)c(M(Z)) \cdot \left(\frac{c_s}{F(s)}-\frac{c_z}{F(z)} \right)  \right)
\end{equation}
for all $z \in Z$ and $s \in Y$. Furthermore, the following hold:
\begin{enumerate}
\item\label{6.1} For any $\{v_{m}\}$ associated with $H(P(v))$ and any subsequence of $\{v_{m}-\log(\lambda)c(M(Z))\}$ associated to some $t' \in N_{\infty}$, we have that $t=t'$.
\item\label{6.2} $t_{z}^{1/F(z)} \geq t_{s}^{1/F(s)}$ for all $s \in {Y}$ and all $z \in Z$.
\item\label{6.4} There is a sequence $\{w_{m}\} \subseteq M(Z)$ associated with $t$.
\item\label{6.5} Assume that $v \in \mathrm{Int}(M(Z))$ and let $z \in Z$. Then $\{ s \in {Y} \ : \ t_{s}^{1/F(s)}=t_{z}^{1/F(z)}\} =Z$.
\end{enumerate}
\end{lemma}
\begin{proof} It follows from \eqref{4.2} of the induction conditions that there is a sequence $\{v_{m}\}$ which is associated with $H(P(v))$. By considering a sub-sequence we can assume that $\{v_{m}-\log(\lambda)c(M(Z))\}$ is associated with some $t \in N_{\infty}$. Choose $Z'$ maximal such that $P(v) \in M(Z')$ with $\text{dim}(M(Z')) \leq k-1$. Then $Z \subseteq Z'$ and Lemma \ref{obsH} implies that $H(P(v))_z \neq 0$ for all $z \in Z$. Hence
\begin{align*}
&\frac{H(P(v))_{s}^{1/F(s)}}{H(P(v))_{z}^{1/F(z)}} = \lim_{m} \left(e^{-\beta(v_{m})F(s)+v_{m}\cdot c_s}\right)^{1/F(s)} \left(e^{-\beta(v_{m})F(z)+v_{m}\cdot c_z}\right)^{-1/F(z)} \\
&=\lim_{m} \exp\left(v_{m}\cdot \left(\frac{c_s}{F(s)}   -  \frac{c_z}{F(z
)}\right)\right)
\end{align*}
for all $z \in Z$, $s \in Y$. Set $w_m =v_{m}-\log(\lambda)c(M(Z))$, and note that
\begin{align*}
&\left(e^{-\beta(w_{m})F(s)+w_{m}\cdot c_s}\right)^{1/F(s)} \left(e^{-\beta(w_{m})F(z)+ w_{m}\cdot c_z}\right)^{-1/F(z)} \\
&=\exp\left(-\log(\lambda)c(M(Z))\cdot \left(\frac{c_s}{F(s)}   -  \frac{c_z}{F(z)}\right)\right)  \exp\left(v_{m}\cdot \left(\frac{c_s}{F(s)}   -  \frac{c_z}{F(z)}\right)\right) 
\end{align*}
for all $s \in Y$ and $z \in Z$. Considering the limit $m \to \infty$ in the last equation for a $s \in Y$ with $t_{s} \neq 0$ and a $z\in Z$, we see that $t_{z} \neq 0$ for $z \in Z$. Hence for all $s \in Y$, $z \in Z$ we get \eqref{33} by combining the two equations,
proving the existence of $t$. The uniqueness follows from Lemma \ref{forhold} since \eqref{33} implies that $t_z \neq 0$ for all $z \in Z$. To establish the additional properties, notice that $(1)$ follows from the above. \eqref{6.4} follows from \eqref{6.1} and the definition of $w_m$ since we can choose $v_m \in M(Z)$, and \eqref{6.2} follows from \eqref{6.4} and Lemma \ref{niceface}. The inclusion $"\supseteq"$ in \eqref{6.5} follows from \eqref{6.2}. For the opposite inclusion in \eqref{6.5}, observe that $c(M(Z)) \cdot (c_s/F(s)-c_z/F(z))<0$ for $s \notin Z$ by \eqref{1.3} of Lemma \ref{l204}. Furthermore, $\log(\lambda) < 0$ since $v \in \mathrm{Int}(M(Z))$ and it follows from Lemma \ref{obsH} that $H(P(v))_{s}^{1/F(s)} \leq H(P(v))_{z}^{1/F(z)}$. Hence \eqref{33} shows that $t_{s}^{1/F(s)}/t_{z}^{1/F(z)}\neq 1$ for $s \notin Z$.

\end{proof}

For $v \in \left(S^{n-1}\cap \mathrm{Int}(M(Z))\right) \backslash \{c(M(Z))\}$ we set $H(v) = t$, where $t \in N_{\infty}$ is the element determined by Lemma \ref{kdef1403}.

\begin{lemma} \label{contk}
$H$ is continuous on $M(Z)\cap S^{n-1}$.
\end{lemma}

\begin{proof} Let $\{x_m\}$ be a sequence such that $\lim_m x_m = x $ in $M(Z)\cap S^{n-1}$. To prove that $\lim_{m \to \infty} H(x_m) = H(x)$, consider a subsequence $\{m_i\}$ such that $\lim_{i \to \infty} H(x_{m_i}) = u$ in $N_{\infty}$. It suffices to show that $u = H(x)$. Assume first that $x \neq c(M(Z))$. By using Lemma \ref{CONT0} it follows that
$$
x = \frac{(1-\lambda)c(M(Z)) + \lambda P(x)}{\left\|(1-\lambda)c(M(Z)) + \lambda P(x)\right\|},
$$
where $\lambda = \lim_{i \to \infty} \lambda_{x_{m_i}}$, and from \eqref{33} by using the continuity of $H$ on the $(k-1)$-skeleton, that
\begin{equation*}
u_{s}^{1/F(s)} = u_{z}^{1/F(z)} \frac{H(P(x))_{s}^{1/F(s)}}{H(P(x))_{z}^{1/F(z)}} \exp\left(-\log(\lambda)c(M(Z)) \cdot \left(\frac{c_s}{F(s)}-\frac{c_z}{F(z)} \right)  \right)
\end{equation*}
for all $s \in Y$ and all $z \in Z$. Hence $u = H(x)$ by uniqueness in Lemma \ref{kdef1403}. Assume then that $x = c(M(Z))$. We may then assume that $x_{m_i} \neq c(M(Z))$ for all $i$, and \eqref{6.2} in Lemma \ref{kdef1403} then implies that $u_{z}^{1/F(z)} \geq u_{s}^{1/F(s)}$ for all $s \in {Y}$ and all $z \in Z$. Hence $u_{s_{1}}^{1/F(s_{1})} = u_{s_{2}}^{1/F(s_{2})}$ when $s_1,s_2 \in Z$. Let $s \in Y \backslash Z$. To conclude from Lemma \ref{l401} that $u = H(c(M(Z)))$ we need only show that $u_s = 0$. Let $z \in Z$. Then 
$$
\lim_{i \to \infty} \frac{H(x_{m_i})_s^{1/F(s)}}{H(x_{m_i})_z^{1/F(z)}} =  \frac{u_s^{1/F(s)}}{u_z^{1/F(z)}}
$$
while
$$
\frac{H\left(\left(P(x_{m_i})\right)\right)_s^{1/F(s)}}{H\left(\left(P(x_{m_i})\right)\right)_z^{1/F(z)}} \leq 1
$$
for all $i$ by Lemma \ref{obsH}. Furthermore, as observed in the proof of Lemma \ref{kdef1403}, $c(M(Z)) \cdot (c_s/F(s)-c_z/F(z))<0$ since $s \notin Z$. Since $\lim_{i \to \infty} \lambda_{x_{m_i}} = 0$ it follows that
$$
 \lim_{i \to \infty} \exp\left((-\log(\lambda_{x_{m_i}}))c(M(Z)) \cdot \left(\frac{c_s}{F(s)}-\frac{c_z}{F(z)} \right)  \right) = 0 .
 $$
 Therefore, by inserting $x_{m_i}$ for $v$ in \eqref{33} and taking the limit $i \to \infty$, it follows that $u_s =0$ as desired.
\end{proof}

We now have a continuous function $H:S^{n-1} \cap M(Z) \to N_{\infty}$ for every $k$-dimensional polyhedral cone $M(Z)$. If $M(Z)$ and $M(Z')$ are two distinct $k$-dimensional polyhedral cones such that $M(Z)\cap M(Z') \cap S^{n-1} \neq \emptyset$, the elements of $M(Z) \cap M(Z') \cap S^{n-1}$ will lie in the $(k-1)$-skeleton and hence the two $H$-functions, one arising from $M(Z)$ and the other from $M(Z')$, will agree on $M(Z) \cap M(Z') \cap S^{n-1}$. In this way we have a well-defined function $H$ defined on the $k$-skeleton. $H$ satisfies the induction condition \eqref{4.1} by Lemma \ref{contk}, \eqref{4.2} by \eqref{6.4} in Lemma \ref{kdef1403} and \eqref{4.3} by \eqref{6.5} and \eqref{6.2} in Lemma \ref{kdef1403}. Since the $n$-skeleton is all of $S^{n-1}$ by Lemma \ref{nskel}, it follows that we have defined a continuous map $H : S^{n-1} \to N_{\infty}$.

\subsubsection{$H$ is a homeomorphism} It remains to show that $H$ is injective and surjective.

\begin{lemma} \label{contk2}
$H : S^{n-1} \to N_{\infty}$ is injective.
\end{lemma}

\begin{proof}
 Note first that it follows from Lemma \ref{l401} that $H$ is injective on the $1$-skeleton. Assume then that $H$ is injective on the $(k-1)$-skeletion, $k \leq n$. Let $x_1,x_2$ be elements in the $k$-skeleton such that $H(x_1) = H(x_2)$. Choose $y \in Y$ such that $H(x_1)_s^{1/F(s)} \leq H(x_1)_y^{1/F(y)}$ for all $s \in Y$. Set 
$$
Z_0 = \left\{ s \in Y: \ H(x_1)_s^{1/F(s)} = H(x_1)_y^{1/F(y)} \right\}.
$$
It follows from the induction condition \eqref{4.3} that $x_1,x_2 \in \Int(M(Z_0))$ and that $Z_0$ is maximal for $M(Z_0)$. If $\dim (M(Z_0)) < k$ it follows from the induction hypothesis that $x_1=x_2$ so we assume that $\dim (M(Z_0)) =k$. If $x_i \neq c(M(Z_0))$ it follows from \eqref{33} that there is an $s \notin Z_0$ such that $H(x_i)_s \neq 0$. Indeed, $P(x_i) \in M(Z')$ with $Z_0 \subsetneq Z'$ and by (3) from the induction conditions $H(P(x_i))_s \neq 0$ when $s \in Z'$. It follows from \eqref{33} that $H(x_i)_s \neq 0$ when $s \in Z' \backslash Z_0$. Therefore, if $H(x_1)_s = 0$ for all $s\notin Z_0$ it follows that $x_1 = x_2 = c(M(Z_0))$. We may therefore assume that $x_i \neq c(M(Z_0)), \ i = 1,2$.  
Note that there is a face $M(Z'_i) \subseteq M(Z_0)$ such that $Z_0 \subsetneq Z'_i$ and $P(x_i) \in M(Z'_i)$. Then $Z_0 \subseteq Z'_1 \cap Z'_2$ and in particular, $y \in Z'_1 \cap Z'_2$.  By symmetry we may assume that $\lambda_{x_1} \leq \lambda_{x_2}$. It follows then from \eqref{33} that
\begin{equation}\label{e1703}
 \frac{H(P(x_1))_{s}^{1/F(s)}}{H(P(x_1))_{y}^{1/F(y)}} \exp \left(\log(\lambda_{x_2}/\lambda_{x_1}) c(M(Z_0)) \cdot \left(\frac{c_s}{F(s)}-\frac{c_y}{F(y)} \right)\right) = \frac{H(P(x_2))_{s}^{1/F(s)}}{H(P(x_2))_{y}^{1/F(y)}} 
\end{equation}
for all $s \in Y$. Consider an element $s_2 \in Z'_2 \backslash Z_0$.
Since $y\in Z'_1$, 
$$
\frac{H(P(x_1))_{s_2}^{1/F(s_2)}}{H(P(x_1))_{y}^{1/F(y)}} \leq 1
$$
and since $s_2 \in Z'_2$,
$$
\frac{H(P(x_2))_{s_2}^{1/F(s_2)}}{H(P(x_2))_{y}^{1/F(y)}} = 1.
$$
Furthermore, it follows from \eqref{1.3} in Lemma \ref{l204} that 
$c(M(Z_0))\cdot \left(\frac{c_{s_2}}{F(s_2)}-\frac{c_y}{F(y)} \right) < 0$,
and \eqref{e1703} therefore implies that $\lambda_{x_1} = \lambda_{x_2}$. It follows then first from \eqref{e1703} and Lemma \ref{forhold} that $H(P(x_1)) = H(P(x_2))$, and then by the induction hypothesis that $P(x_1) = P(x_2)$. Hence $x_1 = x_2$.
\end{proof}

To prove that $H$ is surjective we use the following

\begin{defn} Let $M(Z)$ be a polyhedral cone.
We say that $H$ is \emph{surjective on $M(Z)$} if for every $t \in N_{\infty}$ associated to a sequence $\{v_{m}\}_{m \in \mathbb{N}} \subseteq  M(Z)$, there is some $x \in S^{n-1} \cap M(Z)$ with $H(x)=t$. We then say $H$ is \emph{surjective on the $k$-skeleton} when it is surjective on all polyhedral cones of dimension $ \leq k$.
\end{defn}

\begin{lemma}\label{sur1}
$H$ is surjective on the $1$-skeleton.
\end{lemma}
\begin{proof}
Let $Z\subseteq {Y}$ satisfy that $M(Z)$ is $1$-dimensional. Then $S^{n-1}\cap M(Z)=\{v\}$. Assume that some $t \in N_{\infty}$ has an associated sequence $\{v_{m}\}_{m \in \mathbb{N}} \subseteq M(Z)$. Then $v_{m}/\lVert v_{m} \rVert = v$ for all $m$ and hence
$v_{m}= \lVert v_{m} \rVert v$. It follows then from Lemma \ref{l401} that $t =H(v)$. 
\end{proof}

\begin{lemma}\label{sur2}
Let $k>1$. If $H$ is surjective on the $(k-1)$ skeleton, it is also surjective on the $k$-skeleton.
\end{lemma}
\begin{proof} Let $M(Z)$ be a $k$-dimensional polyhedral cone with $Z$ maximal. Consider an element $t \in N_{\infty}$ with a sequence $\{v_{m}\} \subseteq M(Z)$ associated to $t$. If $v_{m} /\lVert v_{m} \rVert = c(M(Z))$ for infinitely many $m$ it follows from Lemma \ref{l401} as in the proof of Lemma \ref{sur1} that $t=H(c(M(Z)))$. By passing to a subsequence we may therefore assume that $v_{m} /\lVert v_{m} \rVert \neq c(M(Z))$ for all $m$. Since $M(Z)$ has a boundary consisting of a finite set of facets, we can by possibly going to a subsequence assume that there is some $Z' \supsetneq Z$ with $M(Z')$ a facet of $M(Z)$, and $P(v_{m} /\lVert v_{m} \rVert) \in M(Z')$ for all $n$. Set $\lambda_{m} = \lambda_{v_{m}/\lVert v_{m} \rVert}$ and $a_{m} =P(v_{m} /\lVert v_{m} \rVert)$. Then 
$$
\frac{v_{m}}{\lVert v_{m} \rVert} =\frac{(1-\lambda_{m})c(M(Z))+\lambda_{m}a_{m}}{\lVert (1-\lambda_{m})c(M(Z))+\lambda_{m}a_{m} \rVert} .
$$
Setting 
$$
q_{m} := \frac{\lVert v_{m} \rVert}{\lVert (1-\lambda_{m})c(M(Z))+\lambda_{m}a_{m} \rVert} 
$$
we have that 
$$
v_{m}= q_{m}(1-\lambda_{m})c(M(Z))+q_{m}\lambda_{m}a_{m} \qquad \forall m \in \mathbb{N} .
$$
We divide the proof into two cases.

\emph{Case 1: Assume $\{q_{m}(1-\lambda_{m})\}_{m}$ is bounded.} Passing to a subsequence we can then assume that there is a $q\geq 0$ with $\lim_{m}q_{m} (1-\lambda_{m})=q$. Note that $\lambda_{m} \to 1$ since $q_{m} \to \infty$. We can therefore assume that $q_{m}\lambda_{m}a_{m}$ is associated with a $\tilde{t} \in N_{\infty}$. Since $q_{m}\lambda_{m}a_{m}\in M(Z')$ for all $n$, there is some $x \in M(Z') \cap S^{n-1}$ with $H(x)=\tilde{t}$ by our induction hypothesis. Set $\lambda := \exp\left( -q  \right)$ and consider
$$
u:= \frac{(1-\lambda)c(M(Z)) + \lambda x}{\lVert (1-\lambda)c(M(Z)) + \lambda x \rVert} \in M(Z) \cap S^{n-1}.
$$
Then $P(u)=x$ and by \eqref{6.1} of Lemma \ref{kdef1403} $H(u)$ is associated with some subsequence of $w_{m}:=q_{m}\lambda_{m}a_{m}+qc(M(Z)) \in M(Z)$. Fix $z\in Z$.  For any $s \in {Y}$ we have that
\begin{align} \label{ecase1}
\nonumber &\frac{1}{F(s)}\left(-\beta(v_{m})F(s)+v_{m}\cdot c_s  \right)-\frac{1}{F(z)}\left(-\beta(v_{m})F(z)+v_{m}\cdot c_z  \right) = v_{m} \cdot \left( \frac{c_s}{F(s)} -\frac{c_z}{F(z)} \right) \\
&=q_{m}\lambda_{m}a_{m} \cdot \left( \frac{c_s}{F(s)} -\frac{c_z}{F(z)} \right)+q_{m}(1-\lambda_{m})c(M(Z)) \cdot \left( \frac{c_s}{F(s)} -\frac{c_z}{F(z)} \right)
\end{align} 
and
\begin{align} \label{ecase2}
\nonumber &\frac{1}{F(s)}\left(-\beta(w_{m})F(s)+w_{m}\cdot c_s  \right)-\frac{1}{F(z)}\left(-\beta(w_{m})F(z)+w_{m}\cdot c_z  \right) = w_{m} \cdot \left( \frac{c_s}{F(s)} -\frac{c_z}{F(z)} \right) \\
&=q_{m}\lambda_{m}a_{m} \cdot \left( \frac{c_s}{F(s)} -\frac{c_z}{F(z)} \right)+qc(M(Z)) \cdot \left( \frac{c_s}{F(s)} -\frac{c_z}{F(z)} \right)
\end{align}
for all $m$. Note that $t_z \neq 0\neq H(u)_z$ by Lemma \ref{niceface}. Since $q_m(1-\lambda_m) \to  q$ it follows from \eqref{ecase1} and \eqref{ecase2} that
\begin{equation}
\frac{t_{s}^{1/F(s)}}{t_{z}^{1/F(z)}}=\frac{H(u)_{s}^{1/F(s)}}{H(u)_{z}^{1/F(z)}} \quad \forall s \in {Y} .
\end{equation}
Then Lemma \ref{forhold} implies that $H(u)=t$.

\emph{Case 2: Assume $\{q_{m}(1-\lambda_{m})\}_{m}$ is unbounded.}  Passing to a subsequence we can assume that this sequence diverges to $+\infty$. Fix $z \in Z$ and let $s \notin Z$. Then
$$
c(M(Z)) \cdot \left( \frac{c_s}{F(s)}-\frac{c_z}{F(z)} \right) <0\quad \text{ and } \quad a_{m} \cdot \left( \frac{c_s}{F(s)}-\frac{c_z}{F(z)} \right) \leq 0 
$$
and hence
\begin{align} \label{ecase12}
\nonumber &\frac{1}{F(s)}\left(-\beta(v_{m})F(s)+v_{m}\cdot c_s  \right)-\frac{1}{F(z)}\left(-\beta(v_{m})F(z)+v_{m}\cdot c_z  \right)\\
&=q_{m}\lambda_{m}a_{m} \cdot \left( \frac{c_s}{F(s)} -\frac{c_z}{F(z)} \right)+q_{m}(1-\lambda_{m})c(M(Y)) \cdot \left( \frac{c_s}{F(s)} -\frac{c_z}{F(z)} \right) \to -\infty .
\end{align}
Since $t_z\neq 0$ it follows from \eqref{ecase12} that $t_s = 0$.  It follows from Lemma \ref{niceface} that $t_z^{1/F(z)} = t_s^{1/F(s)}$ when $s \in Z$ and we can therefore now conclude from Lemma \ref{l401} that $t = H(c(M(Z))$.
\end{proof}

Since the $n$-skeleton is all of $S^{n-1}$ it follows from Lemma \ref{sur1} and Lemma \ref{sur2} that $H : S^{n-1} \to N_{\infty}$ is surjective, completing the proof of Theorem \ref{Sn-1}.

\section{Two examples}

\subsection{The Heisenberg group} The Heisenberg group $\mathrm{H}_3$ is the subgroup of $\mathrm{Sl}_3(\mathbb Z)$ of matrices of the form
\begin{equation}\label{A}
\left( \begin{matrix} 1 & a & c \\ 0 & 1 & b \\ 0 & 0 & 1 \end{matrix} \right).
\end{equation}
It is wellknown that $\mathrm{H}_3$ is nilpotent and finitely generated.
The canonical set $Y$ of 6 generators consists of the elements
$$
\left( \begin{matrix} 1 & \pm 1 & 0 \\ 0 & 1 & 0\\ 0 & 0 & 1 \end{matrix} \right), \  \left( \begin{matrix} 1 & 0 & 0 \\ 0 & 1 & \pm 1 \\ 0 & 0 & 1 \end{matrix} \right) \ \text{and} \   \left( \begin{matrix} 1 & 0 & \pm 1 \\ 0 & 1 & 0 \\ 0 & 0 & 1 \end{matrix} \right).
$$
We consider the gauge action on $O_Y(\mathrm{H}_3)$, i.e. we set $F(y) = 1$ for all $y \in Y$. There is a homomorphism $\mathrm{H}_3 \to \mathbb Z^2$ sending the matrix \eqref{A} to $(a,b)$ and the kernel of this homomorphism is the commutator group in $\mathrm{H}_3$. It follows that $\Hom (\mathrm{H}_3,\mathbb R)$ is spanned by $c'_1$ and $c'_2$ where
$$
c'_1 \left( \begin{matrix} 1 & a & c \\ 0 & 1 & b \\ 0 & 0 & 1 \end{matrix} \right) = a \ \text{and} \ c'_2 \left( \begin{matrix} 1 & a & c \\ 0 & 1 & b \\ 0 & 0 & 1 \end{matrix} \right) = b.
$$
It is then straightforward to apply the results of the previous sections to deduce the following.
\begin{enumerate}
\item[$\bullet$] There are no $\beta$-KMS states for the gauge action on $O_Y(\mathrm{H}_3)$ when $\beta < \log 6$.
\item[$\bullet$] There is a unique $\log 6$-KMS state for the gauge action on $O_Y(\mathrm{H}_3)$. 
\item[$\bullet$] When $\beta > \log 6$ the simplex of $\beta$-KMS states for the gauge action on $O_Y(\mathrm{H}_3)$ is affinely homeomorphic to the simplex of Borel probability measures on the circle $S^1$.
\item[$\bullet$] The set of $\mathrm{KMS}_{\infty}$ states for the gauge action on $O_Y(\mathrm{H}_3)$ is a convex set affinely homeomorphic to the simplex of Borel probability measures on the circle $S^1$.
\end{enumerate}

\subsection{The infinite dihedral group}\label{Dinfty} The infinite dihedral group $\mathbb D_{\infty}$ is generated by two elements $a,b$ where $b^2 = 1,\  bab = a^{-1}$. With $Y = \{a,b\}$ the Cayley graph $\Gamma\left(\mathbb D_{\infty}, Y\right)$  is the following graph.


\begin{equation*}
  \begin{gathered}
    \begin{xymatrix}{ \hdots \ar[r] & a^{-3} \ar[d]\ar[r] & a^{-2}
        \ar[d] \ar[r] & a^{-1} \ar[d] \ar[r] & e_0 \ar[r] \ar[d] & a
        \ar[r] \ar[d] & a^2 \ar[d]
        \ar[r]  & a^3 \ar[r] \ar[d] & \hdots \\
        \hdots & \ar[l] a^{-3}b \ar@/^/[u] & \ar[l] a^{-2}b \ar@/^/[u]
        & \ar[l] a^{-1}b \ar@/^/[u] & \ar[l] b \ar@/^/[u] & \ar[l] ab
        \ar@/^/[u] & \ar[l] a^2b \ar@/^/[u] & \ar[l] a^{3}b \ar@/^/[u]
        & \ar[l]\hdots }\end{xymatrix}
  \end{gathered}
\end{equation*}

\bigskip
We consider the gauge action on $O_Y\left(\mathbb D_{\infty}\right)$, i.e. we set $F(a) = F(b) = 1$. When $\psi : \mathbb D_{\infty} \to \mathbb R$ is a vector, set
\begin{equation}\label{xyseq}
\psi_{a^n} = x_n\ \text{and} \ 
\psi_{a^nb} = y_n, \ n \in \mathbb Z.
\end{equation}
Then $\psi$ is a $\beta$-harmonic vector iff 
\begin{enumerate}
\item[a)] $x_n \geq 0$, \ $y_n \geq 0$,
\item[b)] $e^{\beta} x_n = x_{n+1} + y_n$, and  
\item[c)] $e^{\beta}y_n = x_n  + y_{n-1}$   
\end{enumerate}
for all $n \in \mathbb Z$. Note that b) and c) are equivalent to b') and c'), where
\begin{enumerate}
\item[b')] $e^{\beta} x_n = x_{n+1} + x_{n-1}$  
\end{enumerate}
and
\begin{enumerate}
\item[c')] $y_{n+1} = x_n$.   
\end{enumerate}
The positive solutions to b') are exactly the $\beta$-harmonic functions on $\mathbb Z$ when we choose $Y = \{-1,1\}$ and $F(-1) = F(1) = 1$ and can be found using the results from Section \ref{AbelKMS}. In the notation used in that section, and with $c'_1 \in \Hom (\mathbb Z,\mathbb R)$ the identity map, we see that the $u(\beta)$ of Lemma \ref{uniqueu} is the solution to the equation
$e^{u(\beta) - \beta} - e^{-u(\beta) - \beta} = 0$, i.e. $ u(\beta) = 0$. We are in Case \ref{case1} when $e^{u(\beta) -  \beta} + e^{-u(\beta)-\beta} = 2e^{-\beta} > 1$, i.e. when $\beta < \log 2$, in Case \ref{case2} when $\beta = \log 2$ and in Case \ref{case3} when $\beta > \log 2$.
When $\beta = \log 2$, $Q(\beta)$ contains only the zero homomorphism which means that $x_n = 1$ for all $n \in \mathbb Z$ and hence also $y_n = 1$ for all $n \in \mathbb Z$ according to c'). The corresponding $\log 2$-KMS measure on $\{a,b\}^{\mathbb N}$ is the Bernoulli measure corresponding to the probability measure $m_0$ on $\{a,b\}$ such that $m_0(\{a\}) = m_0(\{b\}) = 1/2$. 
When $\beta > \log 2$,
$$
Q(\beta) \simeq \left\{ c \in \mathbb R: \ e^{c } + e^{-c} = e^{\beta} \right\} = \left\{-c(\beta), c(\beta)\right\},
$$
where $c(\beta)  > 0$ is the positive solution to $e^c + e^{-c} = e^{\beta}$. It follows therefore from Theorem \ref{Q-descr} and c') that the set of $\beta$-harmonic vectors on $\mathbb D_{\infty}$ are parametrized by the interval
$\left[0,1\right]$ such that $0 \leq t \leq 1$ corresponds to the solution
\begin{equation}\label{solutions}
x_n = te^{c(\beta)n} + (1-t)e^{-c(\beta) n}, \ \ y_n = te^{c(\beta)(n-1)} + (1-t)e^{-c(\beta) (n-1)} \ 
\end{equation}
for all $n \in \mathbb Z$. It follows that for $\beta > \log 2$ the simplex of normalized $\beta$-KMS states is affinely homeomorphic to $[0,1]$. By Proposition \ref{Gabbb} the only abelian KMS state is the $\log 2$-KMS measure.

It is not difficult to identify the set of $\mathrm{KMS}_{\infty}$ states as a convex set with two extreme points; limits of the two extreme $\beta$-KMS states as $\beta \to \infty$. The Borel probability measures on $Y^{\mathbb N} = \{a,b\}^{\mathbb N}$ corresponding to the extreme $\mathrm{KMS}_{\infty}$ states are the Dirac measures at $a^{\infty} \in \{a,b\}^{\mathbb N}$ and $ba^{\infty} \in \{a,b\}^{\mathbb N}$; the two geodesic paths emitted from $e_0$ in the graph above.

\end{document}